\documentclass[12pt]{amsart}

\usepackage[english]{babel}
\usepackage[alphabetic]{amsrefs}
\usepackage{graphicx}
\graphicspath{ {images/} }
\usepackage[T1]{fontenc}
\usepackage{txfonts}
\usepackage{times}
\usepackage{amssymb,verbatim,amscd,amsfonts,amsmath,amsthm,enumerate}
\usepackage[all]{xy}
\usepackage{subfig}
\usepackage{tikz}
\usetikzlibrary{shapes,arrows,shadows}
\usetikzlibrary{decorations.markings}
\usepackage{color}
\usepackage[normalem]{ulem}
\usepackage{hyperref}
\usepackage{wrapfig}
\usetikzlibrary{arrows}
\usepackage{pb-diagram}

\newtheorem{teorema}{Theorem}[section]
\newtheorem{lema}{Lemma}[section]
\newtheorem{proposicion}{Proposition}[section]
\newtheorem{corolario}{Corollary}[section]

\newtheorem{pregunta}{Question}[section]

\theoremstyle{definition}
\newtheorem{construccion}{Construction}[section]
\newtheorem{definicion}{Definition}[section]
\newtheorem{example}{Example}[section]
\newtheorem{remark}{Remark}[section]

\newcommand{\R}{\mathbb{R}}

\def\Sing{\operatorname{Sing}}
\def\Aff{\operatorname{Aff}}

\title{Veech groups of infinite genus surfaces}
\author{Camilo Ram\'irez Maluendas and Ferr\'an Valdez}
\date{\today}

\address{
Camilo Ram\'irez Maluendas \newline
FUNDACI\'ON UNIVERSITARIA KONRAD LORENZ.\newline
CP. 110231, BOGOT\'A, COLOMBIA.
}
\email{camilo.ramirezm@konradlorenz.edu.co}
\address{
Ferr\'an Valdez \newline
Centro de Ciencias Matem\'aticas, UNAM, Campus Morelia, C.P. 58190, Morelia, \newline 
Michoac\'an,  M\'exico.
}
\email{ferran@matmor.unam.mx}

\begin{document}
\maketitle
\begin{abstract}
We show that every countable subgroup $G<\rm GL_+(2,\mathbb{R})$ without contracting elements is the Veech group of a tame translation surface $S$ of infinite genus, for infinitely many different topological types of $S$. Moreover, we prove that as long as every end has genus, there are no restrictions on the topological type of $S$ to realise all possible \emph{uncountable}  Veech groups. 
\end{abstract}



\section*{Introduction}

A surface $S$ endowed with an atlas whose transition functions are translations is called \emph{a translation surface}. To each such surface we can associate the group $\rm Aff_+(S)$ formed by all orientation-preserving homeomorphisms of $S$ which are affine in local coordinates. It is easy to see that the derivative of any affine homeomorphism of $S$ is constant, hence we have a well defined map 
\begin{equation}
	\label{MapDerAffSurf}
	\rm Aff_+(S)\stackrel{D}{\longrightarrow}GL_+(2,\mathbb{R}).
\end{equation}
The main purpose of this article is the study of the image $\Gamma(S)$ of this map. 
When $S$ is a compact translation surface\footnote{That is, when the metric completion of $S$  with respect to the natural flat metric is a compact surface of genus $g\geq 1$.} $\Gamma(S)$ lies in $\rm SL(2,\mathbb{R})$, is Fuchsian and hence acts on the hyperbolic plane by M\"{o}bius transformations. In a milestone paper \cite{Vee}, W.A. Veech proved the geodesic flow on a compact translation surface $S$ for which $\Gamma(S)$ is a lattice behaves, roughly speaking, like the geodesic flow on a flat torus. For this reason it is usual to call $\Gamma(S)$  \emph{the Veech group} of the surface $S$. Aside from questions regarding the dynamical properties of the geodesic flow, it is natural to investigate which subgroups of $\rm GL_+(2,\mathbb{R})$ can be realised as Veech groups. For compact translation surfaces, this is a difficult question. Moreover, simpler instances of this problem, as the existence of hyperbolic cyclic Veech groups, are still open. 

In this paper we address the aforementioned realisation problem from the perspective of  \emph{infinite type} tame translation surfaces\footnote{These are surfaces whose fundamental group is not finitely generated and whose natural flat metric has only singularities of (possibly infinite) conic type. Compact translation surfaces are tame, but not all translation surfaces of infinite type are tame. We refer the reader to \cite{BV} for a general discussion on singularities of translation surfaces.}. This kind of translation surfaces appears naturally when studying irrational polygonal billiards \cite{Va} or infinite coverings of compact translation surfaces \cite{HW}. Given that translation surfaces are orientable, the topological type of an infinite type translation surface is determined by its genus $g=g(S)\in\mathbb{N}\cup\{\infty\}$  and a couple of nested compact, metrisable and totally disconnected spaces $Ends_{\infty}(S)\subset Ends(S)$. Here  $ Ends(S)$ is the space of ends of $S$ and ${\rm Ends}_{\infty}(S)$ is formed by those ends that carry (infinite) genus. Reciprocally, every couple of nested closed subspaces of the Cantor set $X_\infty\subset X\subset 2^{\omega}$
 can be realised as the space of ends and the space of ends with genus of some surface. The following result describes in general terms what is to be expected from the Veech group of a tame translation surface.
\begin{teorema}\cite{PSV}
\label{T:PSV1}
 The Veech group $\Gamma(S)$ of a tame translation surface is either:
\begin{enumerate}
  \item Countable and does not contain contracting elements, or
  \item Conjugated to
	$
P:=\left\{
\begin{pmatrix}
1 & t \\
0 & s
\end{pmatrix}\hspace{1mm}:\hspace{1mm} t\in\mathbb{R} \text{ y }
\hspace{1mm} s\in\mathbb{R}^{+}
\right\},\hspace{1mm}\text{or}
$
  \item Conjugated to $P'<{\rm GL_{+}(2, \R)}$, the subgroup generated by $P$ and $\rm -Id$, or
  \item Equal to ${\rm GL_{+}(2, \R)}$.
\end{enumerate}
\end{teorema}
It is not difficult to see that condition 4 above implies that $S$ is isometric to either the plane or a ramified covering of the plane. Hence, it is natural to ask if all other expected groups can be realised within \emph{the same topological class} of an infinite type tame translation surface:


\begin{pregunta}
\label{Q:1}
Let $X_\infty\subset X\subset 2^\omega$ be  a nested couple of closed subspaces of the Cantor set. Is it possible to realise any subgroup of $\rm GL_+(2,\mathbb{R})$ satisfying (1), (2), or (3) in theorem \ref{T:PSV1} as the Veech group of some \emph{tame} translation surface $S$ satisfying $X_\infty=Ends_\infty(S)$ and $X=Ends(S)$?
\end{pregunta}

For the simplest instance of this question $X_\infty= X =\{*\}$ the answer is positive:
\begin{teorema}\emph{(Ibid.)}
\label{T:PSV2}
Any subgroup of ${\rm GL_{+}(2, \R)}$ satisfying (1), (2), or (3) in theorem \ref{T:PSV1} can be realised as the Veech group of an infinite genus tame translation with only one end.
\end{teorema}

The main contribution of this article is to show that question \ref{Q:1} has a positive answer for a large topological class of infinite type tame translation surfaces. For instance, in \S \ref{SEC:UNCOUNTABLEVEECH} we show that within the topological class of infinite type tame translation surfaces $S$ for which $Ends_\infty(S)=Ends(S)$, there are no restrictions to realise \emph{uncountable} Veech groups:

\begin{teorema}
\label{T:PP}
Let $X$ be any closed subset of the Cantor set and $P, P^{'}<GL_{+}(2,\mathbb{R})$ be as in theorem \ref{T:PSV1}. Then there exist tame translation surfaces $S$ and $S^{'}$ for which:
\begin{itemize}
\item $Ends(S)=Ends_\infty(S)=Ends(S')=Ends_\infty(S')=X$
\item the Veech groups of $S$ and $S'$ are conjugated to $P$ and $P^{'}$, respectively.
\end{itemize}
\end{teorema}

The rest of the results we present deal with realising \emph{countable subgroups of $\rm GL_+(2,\mathbb{R})$ without contracting elements} as Veech groups of infinite type tame translation surfaces satisfying $Ends_\infty(S)=Ends(S)$. We present them according to the following heuristic picture: 
if one was to order infinite genus surfaces satisfying $Ends_\infty(S)=Ends(S)$ according to their topological complexity, the simplest surface would be the one for which the space of ends is just a singleton and the most sophisticated would be the one for which the space of ends is homeomorphic to the Cantor set. These surfaces are called the \emph{Loch Ness Monster} and the \emph{Blooming Cantor Tree}, respectively (see Figures \ref{Figure1} and \ref{Figure2}). The nomenclature is due to Sullivan \cite{PSul} and Ghys \cite{Ghys}. 

\begin{teorema}
\label{T:AF}

Let $G< {\rm GL_{+}(2,\mathbb{R})} $ be any countable subgroup without contracting elements. Then there exist a tame translation surface $S$ homeomorphic to Blooming Cantor tree whose Veech group is $G$.
\end{teorema}

In other words, question \ref{Q:1} has a positive answer if we move from the simplest topological type of an infinite genus surface without planar ends to the most sophisticated one. After this, we show that for infinitely many cases "in between" question  \ref{Q:1} has also a positive answer. 

\begin{teorema}
\label{T:AN}
Let $X$ be a countable closed subspace of the Cantor set with characteristic system $(k,1)$ where $k\in \mathbb{N}\cup\{0\}$, and $G<{\rm GL_{+}(2,\mathbb{R}) }$ any countable subgroup without contracting elements. Then there exists a tame translation surface $S$ whose spaces of ends $Ends_\infty(S)=Ends(S)$ are homeomorphic to $X$  and whose Veech group is $G$.
\end{teorema}
Recall that a countable compact Hausdorff space $X$ has \emph{characteristic system} $(k,n)$ if its $k$-th Cantor-Bendixon derivative, that we denote by $X^k$, is a finite set of $n$ points \cite{MaSi}. We remark that for $k=0$, the statement of theorem \ref{T:AN} is the statement of theorem \ref{T:PSV2} for countable subgroups of $\rm GL_+(2,\mathbb{R})$ without contracting elements.

On the other hand, we know from the Cantor-Bendixson theorem that every uncountable closed subset of the Cantor set $2^\omega$ is homeomorphic to a subset of the form $B\sqcup U\subset2^{\omega}$, where $B$ is homeomorphic to $2^{\omega}$ and $U$ is countable and discrete.

 \begin{teorema}
 \label{T:NN}

Let $B\sqcup U$ be an uncountable closed subset of the Cantor set, where $B$ is homeomorphic to the Cantor set and $U$ is countable, discrete and its boundary $\partial U$ is just one point. Then for any countable  subgroup 
$G<{\rm GL_{+}(2,\mathbb{R})}$ without contracting elements there exist a tame translation surface $S$ whose spaces of ends $Ends_\infty(S)=Ends(S)$ are homeomorphic to $B\sqcup U$ and whose  Veech group is $G$. \end{teorema}

\begin{figure}[h!]
 \centering
 \includegraphics[scale=0.8]{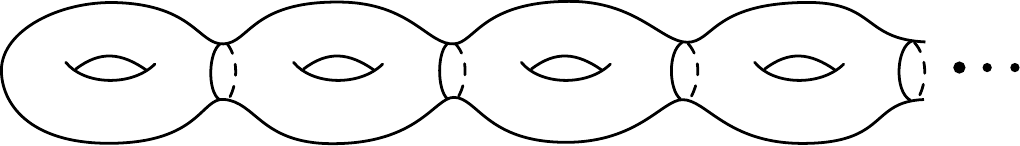}\\
 \caption{\emph{The Loch Ness Monster.}\\}
 \label{Figure1}
\end{figure}

\begin{figure}[h!]
  \centering
  \includegraphics[angle=-90,scale=0.3]{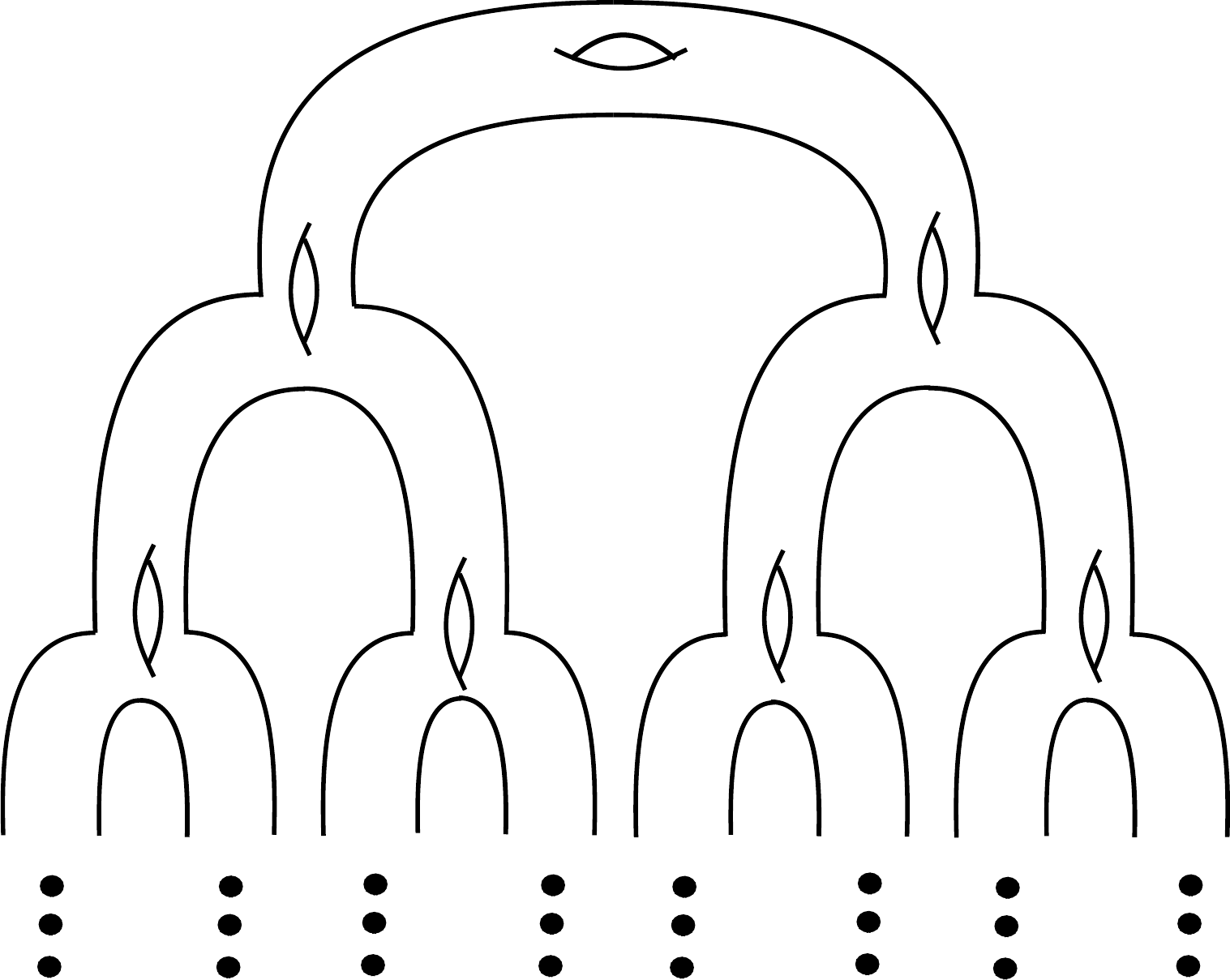}\\
  \caption{\emph{The Blooming Cantor tree.}}
  \label{Figure2}
\end{figure}

The proof of theorem \ref{T:PP} is largely inspired in the proof of theorem \ref{T:PSV2} found in \cite{PSV} and a convenient characterisation of spaces of ends using binary trees that we introduce in the next chapter. On the other hand, theorems \ref{T:AF}, \ref{T:AN} and \ref{T:NN} are consequence of an abstract and general construction that we call \emph{a puzzle} and define in  \S \ref{Section:PUZZLES}. As a matter of fact, some of the main results found in \cite{PSV} can be deduced using this construction. Sadly, puzzles do not work to give an answer to all possible instances of question \ref{Q:1}. For example it is still unknown if every countable subgroup $G<\rm GL_+(2,\mathbb{R})$  can be realised as the Veech group of an infinite type translation surface $S$ whose spaces of ends $Ends_\infty(S)=Ends(S)$ contain only $n\geq 2$ elements. However, given the evidence provided by the results shown in this article we conjecture that question \ref{Q:1} \emph{has always a positive answer}. 
\begin{remark}
All the translation surfaces constructed in this article will have infinite area. We believe that this hypothesis is crucial to answer question \ref{Q:1} positively. For finite area infinite type translation surfaces little is known about the Veech group, see for example \cite{BV}. 
\end{remark}
\textbf{Reader's guide}. In \S \ref{SEC:PRELIMINARIES} we discuss the basic concepts needed to develop the proofs of our main results. Namely, we recall the topological classification of all orientable surfaces, the Cantor-Bendixon derivative and the notion of a characteristic system. We also show that every closed subset $X$ of the Cantor set is the space of ends of a subgraph $T_X$ of an infinite (binary) tree. In particular, we introduce a decomposition of $T_X$ into a countable union of infinite paths which is crucial to the constructions we present. We finish this section by explaining the construction of the "building block" used to define puzzles. In \S \ref{SEC:UNCOUNTABLEVEECH} we prove theorem \ref{T:PP}. Finally, in \S \ref{Section:PUZZLES} we introduce the so called puzzles and show how theorems \ref{T:AF}, \ref{T:AN} and \ref{T:NN} follow from the same general construction. 

\textbf{Acknowledgments}. The authors sincerely thank Jes\'us Muci\~no Raymundo, Fernando Hern\'andez Hern\'andez, Osvaldo Guzm\'an Gonz\'alez, and Ariet Ramos Garc\'ia for their constructive conversations and  valuable help. The first author was partially supported by CONACYT and CCM-UNAM. The second author was generously supported by LAISLA, CONACYT CB-2009-01 127991 and PAPIIT projects IN100115, IN103411 \& IB100212.

\section{Preliminaries}
	\label{SEC:PRELIMINARIES}
We begin this section recalling the topological classification of orientable surfaces, the Cantor-Bendixon derivative and the notion of a characteristic system. We then introduce a model for spaces of ends of surfaces based on subgraphs of an infinite (binary) tree. We finish by explaining the construction of the \emph{elementary piece} that will be used in the next section to define puzzles and give the proof of  theorems \ref{T:AF}, \ref{T:AN} and \ref{T:NN}.

\subsection{Topological classification of surfaces.}
Through this text we will only work with \emph{orientable} surfaces $S$ without boundary. It is a well know fact that any such surface with finitely generated fundamental group is determined up to homeomorphism by its genus and number of punctures. When the fundamental group of the surface is not finitely generated new topological invariants, namely \emph{the spaces of ends of $S$}, are needed to determine the surface up to homeomorphism. In what follows we review these invariants in detail, for they will be used in the proof of our main results. For more details we refer the reader to \cite{Ian}. 


\vspace{1mm}
A \textit{pre-end} of a connected surface $S$ is a nested sequence $U_1\supset U_2\supset \cdots$ of connected open subsets of $S$ such that the boundary of $U_n$ in $S$ is compact for every $n\in\mathbb{N}$ and for any compact subset $K$ of $S$ there exist $l\in\mathbb{N}$ such that $U_{l}\cap K=\emptyset$. We shall denote the pre-end $U_1\supset U_2\supset \cdots$ as $(U_n)_{n\in\mathbb{N}}$. Two such sequences $(U_{n})_{n\in\mathbb{N}}$ and $(U_{n}^{'})_{n\in \mathbb{N}}$ are said to be equivalent if for any $i \in \mathbb{N}$ exist $j \in \mathbb{N}$ such that $U_{j}'\subset U_i$ and viceversa. We denote by $Ends(S)$ the corresponding set of equivalence classes and call each equivalence class $[U_{n}]_{n\in\mathbb{N}}\in Ends(S)$ an \textit{end} of $S$. 
We endow $Ends(S)$ with a topology by specifying a pre-basis as follows: for any open subset $W\subset S$ whose boundary is compact, we define $W^{\ast}:=\{[U_{n}]_{n\in \mathbb{N}}\in Ends(S): W\supset U_{l}\hspace{2mm}\text{for $l$ sufficiently large}\}$.  We call the corresponding topological space \emph{the space of ends of S}. 

\begin{proposicion}
\label{p:1.2}
\cite[Proposition 3]{Ian} The space of ends of a connected surface $S$ is totally disconnected, compact, and Hausdorff. In particular, $Ends(S)$ is homeomorphic to a closed subspace of the Cantor set.
\end{proposicion}

A surface is said to be \textit{planar} if all of its compact subsurfaces are of genus zero. An end $[U_n]_{n\in\mathbb{N}}$ is called \textit{planar} if there exist $l\in\mathbb{N}$ such that $U_l$ is planar. The \textit{genus} of a surface $S$ is the maximum of the genera of its compact subsurfaces. Remark that if a surface $S$ has \textit{infinite genus} there exist no finite set $\rm \mathcal{C}$ of mutually non-intersecting simple closed curves with the property that $S\setminus \mathcal{C}$ is  \textit{connected and planar}. We define $Ends_{\infty}(S)\subset Ends(S)$ as the set of all ends of $S$ which are not planar. It follows from the definitions that $Ends_{\infty}(S)$ forms a closed subspace of $Ends(S)$.

\begin{teorema}[Classification of orientable surfaces]
\label{t:1.1}
\cite[Chapter 5]{Ker} Let $S$ and $S^{'}$ be two orientable surfaces of the same genus. Then $S$ and $S^{'}$ are homeomorphic if only if both $Ends_{\infty}(S)\subset Ends(S) $ and $Ends_{\infty}(S^{'})\subset Ends(S^{'})$ are homeomorphic as nested topological spaces.
\end{teorema}

Of special interest in this paper is the surface of infinite genus and only one end and the surface without planar ends whose ends space is homeomorphic to the Cantor set. The former is known as \emph{the Loch Ness Monster} (Figure \ref{Figure1}) and the latter as the \emph{Blooming Cantor tree} (Figure \ref{Figure2}). The nomenclature is due to \cite{PSul} and \cite{Ghys} respectively. Remark that a surface $S$ has only one end if only if for all compact subset $K\subset S$ there exist a compact $K^{'}\subset S$ such that $K\subset K^{'}$ and $S\setminus K^{'}$ is connected, see \cite{Spec}.





\subsection{Cantor-Bendixson derivative}
We recall briefly Cantor-Bendixon's theorem and how the Cantor-Bendixon rank classifies countable compact Polish spaces up to homeomorphism.


\begin{teorema}[Cantor-Bendixson, \cite{Kec}]
\label{t:1.2}
Let $X$ be a Polish space. Then $X$ can be uniquely written as $X=B\sqcup U$, where $B$ is perfect  and $U$ is  countable and open.
\end{teorema}

\begin{corolario}\label{c:1.1}
Any uncountable Polish space contains a closed subset homeomorphic to the Cantor set.
\end{corolario}

Given a topological space $X$, the \textit{Cantor-Bendixson derivate} of $X$ is the set $X^{'}:=\{ x\in X: x \text{ is a limit point of } X\}$. For every $k\in\mathbb{N}$, \emph{the $k$-th Cantor-Bendixon derivative of $X$} is defined as $X^k:=(X^{k-1})'$, where $X^0=X$. A countable Polish space is said to have \emph{characteristic system} $(k,1)$ if $X^k\neq\emptyset$, $X^{k+1}=\emptyset$ and the cardinality of $X^k$ is exactly one.

\begin{remark}
	\label{c:1.2}
Any countable Polish space with characteristic system $(k,1)$ is homeomorphic to the ordinal number\footnote{Recall that ordinal number can be made into a topological space by endowing it with the order topology.} $\omega^k+1$, where $\omega$ is the least infinite ordinal. In particular, every two countable Polish spaces with characteristic system $(k,1)$ are homeomorphic. For further details, see \cite{MaSi}.
\end{remark}

The following results will be used in the proof of theorem \ref{T:NN}.

\begin{teorema}
\label{t:1.3}
Let $X_{1}=B_{1}\sqcup U_{1}$ and $X_{2}=B_{2}\sqcup U_{2}$ be two uncountable closed subset of  the Cantor set. Suppose that both  boundaries $\partial U_{1}$ and $\partial U_{2}$ are just singletons. Then $X_1$ and $X_2$ are homeomorphic.
\end{teorema}

\begin{proof}
The Cantor set $2^{\omega}$ is identified with the topological group $\prod_{i\in\mathbb{N}}\mathbb{Z}_2^i$ (see \cite[p. 50]{Kec}). Recall that this group acts by homeomorphisms (freely and transitively) on itself. Let us denote this action by:
\begin{equation}\label{eq:2}
\begin{array}{ccccl}
 \alpha & : & \prod\limits_{i\in\mathbb{N}} \mathbb{Z}_2^i\times \prod\limits_{i\in\mathbb{N}}\mathbb{Z}_2^i & \to & \prod\limits_{i\in\mathbb{N}}\mathbb{Z}_2^i, \\
            &   &                                                                                                                                                      &        &            \\
            &   &([x_n]_{n\in\mathbb{N}},[y_n]_{n\in\mathbb{N}})                                                 & \to &  [x_n+y_n]_{n\in\mathbb{N}}.
\end{array}
\end{equation}
By hypothesis there are homeomorphisms 
$f_j:B_j\to \prod_{i\in\mathbb{N}}\mathbb{Z}_2^i$, $j=1,2$ and an element $[z_n]_{n\in \mathbb{N}}\in \prod_{i\in\mathbb{N}}\mathbb{Z}_2^i$ such that the homeomorphism $f_{2}^{-1}\circ \alpha_{|[z_n]_{n\in\mathbb{N}}}\circ f_1$ sends $\partial U_1$ to $\partial U_2$. On the other hand, if we denote $\overline{U_j}$, the closure of $U_j$ in $X_j$, $j=1,2$, there exists a homeomorphism $h:\overline{U_1}\to\overline{U_2}$. In particular $h(\partial U_1)=\partial U_2$. Hence
\[
\begin{array}{ccccl}
F & : &  X_1 & \to & X_2 \\
   &   & x     & \to & \begin{cases}
                                    f(x), & \text{if $x\in B_1$,}\\
                                    h(x), & \text{if $x\in \overline{U_1}$,}
                               \end{cases}
\end{array}
\]
is well a defined bijection. We claim that  $F$ is the desired homeomorphism. Let $V$ be an open set in $X_2$ containing $F(\partial U_1)=\partial U_2$. There are open sets $W_1$ and $W_2$ in $X_1$ containing the point $\partial U_1$ such that $f(W_1\cap B_1)\subset V\cap B_2$ and $h(W_2\cap \overline{U_1})\subset V\cap\overline{U_2}$. Then $F(W_1\cap W_2)\subset V$. This proves that function $F$ is \emph{continuous} at $\partial U_1$. Since $F$ is defined by glueing homeomorphisms at this point, this is sufficient to prove that $F$ is continuous. On the other hand, the function $F$ is \emph{closed} because every continuous function from a compact and Hausdorff space onto a Hausdorff space is closed (see \cite[p.226]{Dugu}). Therefore $F$ is a homeomorphism.
\end{proof}

\subsection{The Cantor Binary tree }
	\label{SS:CantorBinaryTree}
One of the fundamental objects we use in the construction of infinite translation surfaces with prescribed Veech group is the infinite 3-regular tree. This graph plays a distinguished role for the space of ends of any surface is homeomorphic to a subspace of its space of ends.  We give binary coordinates to the vertex set of the infinite 3-regular tree, for we use these in a systematic way during the proofs of our main results. 

For every $n\in\mathbb{N}$ let $2^{n}:=\{D_s: D_s\in \prod_{i=1}^n\{0,1\}_i\}$ and $\pi_i:2^n \to \{0,1\}$ be the projection onto the $i$-th coordinate. We define $V:=\{ D_s: D_s\in 2^n \text{ for some } n\in \mathbb{N}\}$ and $E$ as the union of $((0),(1))$ with the set $\{ (D_s,D_t) : D_s\in 2^{n}$ and $D_t\in 2^{n+1}$ for some $n\in\mathbb{N}$, and $\pi_i(D_s) =\pi_i(D_t)$  for every $i\in\{1,...,n\}\}$. The infinite 3-regular tree with vertex set $V$ and edges $E$ will be called the \emph{Cantor binary tree} and denoted by $T2^\omega$, see Figure \ref{Figure3}.


\begin{figure}[h!]
 \centering
 \includegraphics[scale=0.5]{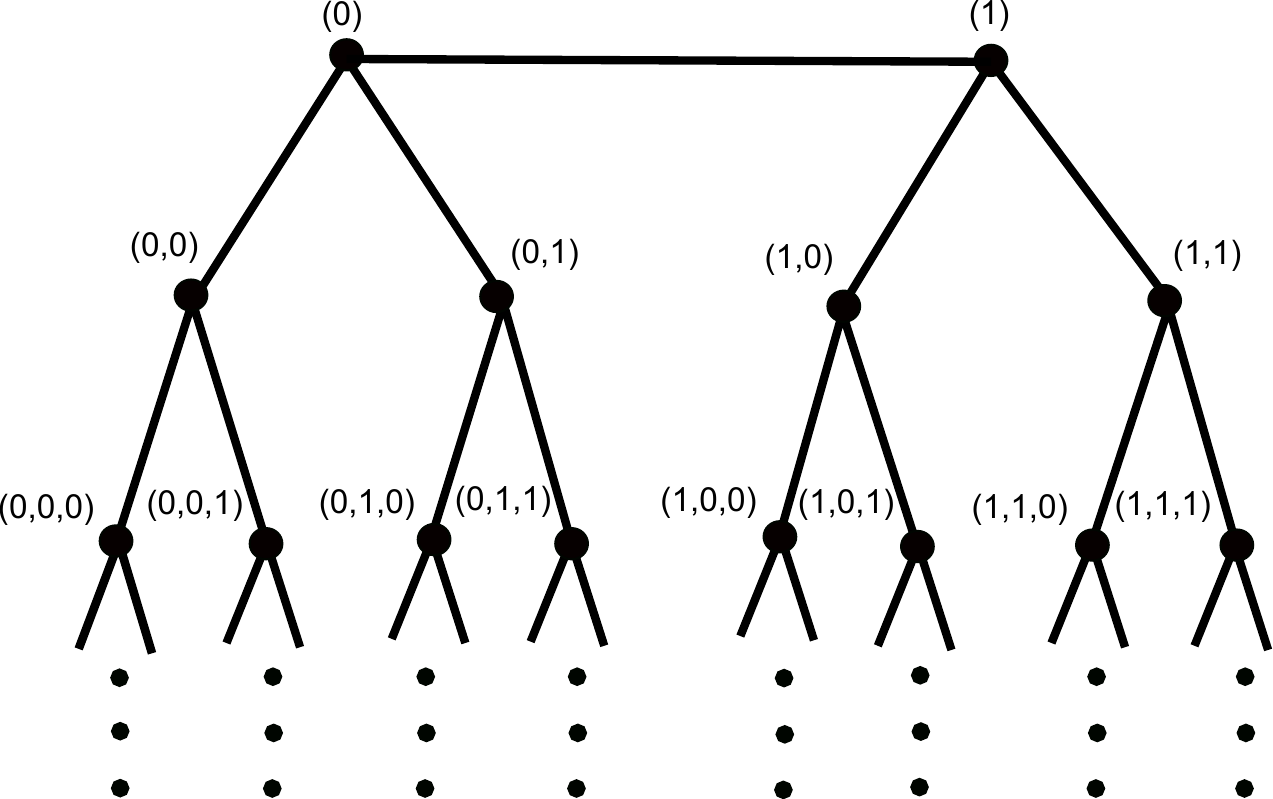}\\
 \caption{\emph{Cantor binary tree  $T2^{\omega}$.}}
 \label{Figure3}
\end{figure}

\begin{remark}
	\label{REM:InfPathsGraph}
Let $(v_n)_{n\in\mathbb{N}}$, $v_n\in 2^\omega$ be an infinite simple path in $T2^\omega$. If we define $V_n$ as the connected component of $T2^{\omega}\setminus\{v_{n}\}$ such that $v_{n+1}\in V_n$, then $[V_n]\in Ends(T2^\omega)$ is completely determined by $(v_n)_{n\in\mathbb{N}}$. Hence, if we endow $\{0,1\}$ and $2^\omega:=\prod\limits_{i\in\mathbb{N}} \{0,1\}_{i} $ with the discrete and product topologies respectively, the map:
\begin{equation}\label{eq:4}
\begin{array}{ccccc}
f  & :  & \prod\limits_{i\in\mathbb{N}} \{0,1\}_{i} & \to &  Ends(T2^{\omega}) \\
    &    & (x_{n})_{n\in\mathbb{N}}             &  \to & ( v_n:=(x_1,\ldots,x_{n})),
\end{array}
\end{equation}
is a homeomorphism between the standard binary Cantor set and the space of ends of $T2^\omega$. 
\end{remark}

\begin{remark}
	\label{r:0:2}
Sometimes we will abuse notation and denote by $f((x_n)_{n\in\mathbb{N}})$	both the end defined by the infinite path $( v_n:=(x_1,\ldots,x_{n}))_{n\in\mathbb{N}}$ and the infinite path $(v_n:=(x_1,\ldots,x_{n}))_{n\in\mathbb{N}}$ in $T2^\omega$.
\end{remark}




\begin{lema}\label{l:1.1}
Let $X$ be a closed subset of the Cantor set. Then there exist a connected subgraph $T_{X}\subset T2^{\omega}$ such that its ends space $Ends(T_X)$ is homeomorphic to $X$.
\end{lema}

\begin{proof}
Without loss of generality we suppose that $X\subset \prod\limits_{i\in\mathbb{N}} \{0,1\}_{i}$. Define:
\begin{equation}
	\label{e:0:4}
T_X:=\left(\bigcup_{(x_n)_{n\in\mathbb{N}}\in X} f((x_n)_{n\in\mathbb{N}})\right)\cup ((0),(1)) \subset T2^{\omega},
\end{equation}
where $f$ is the map defined in (\ref{eq:4}).
\end{proof}

The constructions of translations surfaces that we present follow at most a countable number of steps. In order to integrate the graphs $T_X$ to these constructions it is important to be able to express them as a \emph{countable} union of infinite paths which intersect at most at a vertex. 


\begin{example}
	\label{E:0:1}
Define $\mathfrak{T}$ as the union of $((0), (1),(1,1),(1,1,1),...)$ and $((0),(0,0),(0,0,0),...)$ with the countable set of infinite paths:
$$
\{((D_s), (D_s,x),(D_s,x,x),(D_s,x,x,x),...)\in T2^{\omega} :  D_s\in 2^n, x \in\{0,1\}, \text{ and }  \pi_n(D_s)\neq x\}_{n\in\mathbb{N}},
$$
then it is easy to see that $T2^{\omega}:=\bigcup_{\gamma\in\mathfrak{T}}\gamma$. Moreover, any two infinite paths in $\mathfrak{T}$ are either disjoint or intersect in just one vertex (see Figure \ref{Figure4}).
\end{example}


\begin{figure}[h!]
 \centering
 \includegraphics[scale=0.4]{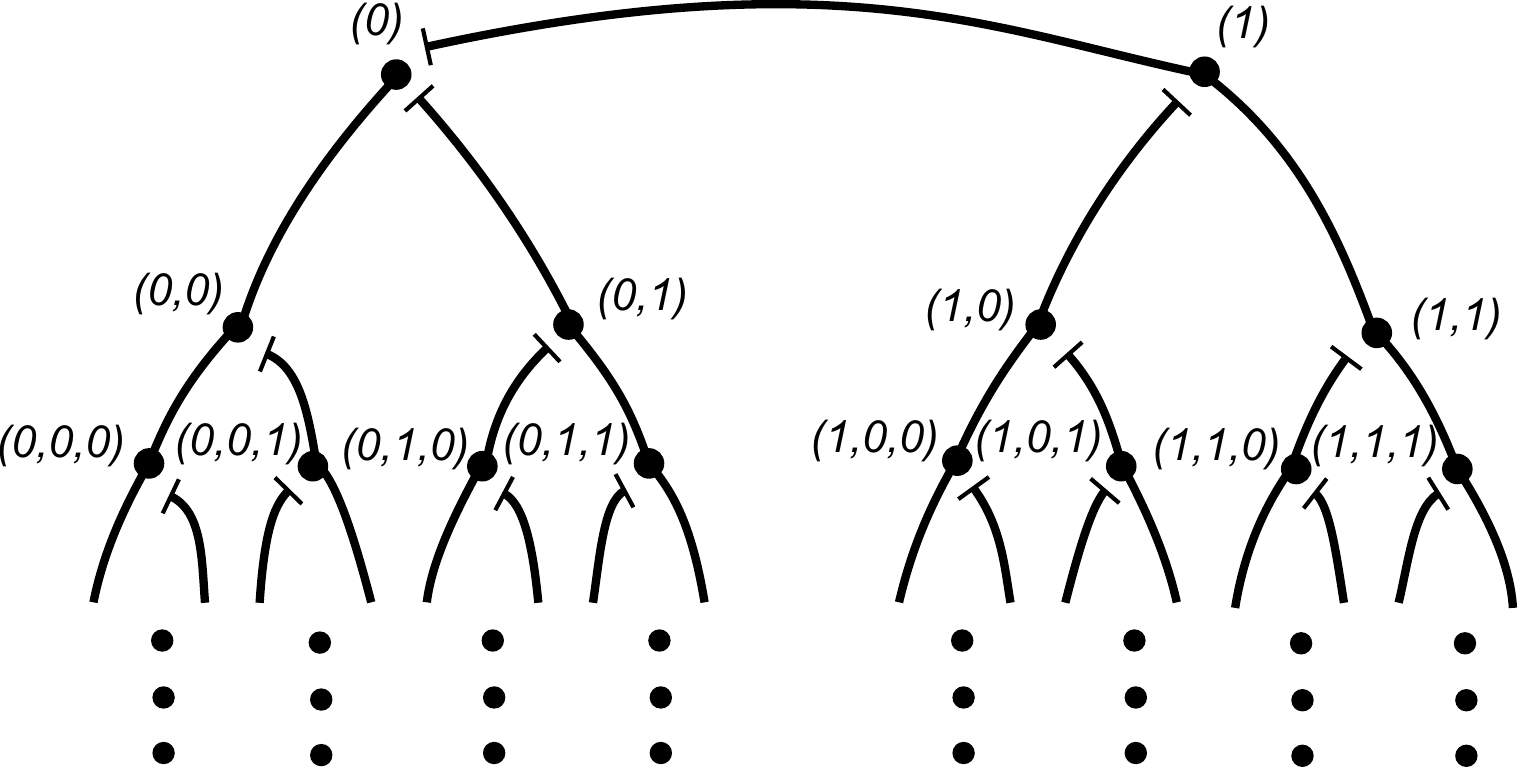}\\
 \caption{\emph{Countable family of infinite paths $\mathfrak{T}$.}}
 \label{Figure4}
\end{figure}

\begin{lema}\label{l:1.2}
Let $X\subset 2^\omega$ be a closed of the Cantor set and let $T_X$ be the subgraph given by lemma \ref{l:1.1}. Then there exist a countable set of infinite paths $\mathfrak{T}_X$ such that $T_{X}=\bigcup_{\gamma\in\mathfrak{T}_X}\gamma$. Moreover, every two infinite paths $\gamma,\gamma'\in\mathfrak{T}_X$ intersect in at most one vertex. 
\end{lema}

\begin{proof} 
Let us first suppose that $X$ is uncountable. By theorem \ref{t:1.2},  $Ends(T_{X})=B\sqcup U$, where $B$ is homeomorphic to the Cantor set and $U$ is a countable,  discrete and open. If $U=\emptyset$, then w.l.o.g. we can suppose that $T_X=T2^\omega$ and treat this case as in example \ref{E:0:1}. Let us suppose then that $U= \{u_{n}\}_{n\in\mathbb{N}}$. By Lemma \ref{l:1.1} we have
\[
T_{X}=\left(\bigcup_{x\in B\sqcup U}f(x)\right)\cup ((0),(1))=\left(\bigcup_{x\in B}f(x)\right)\cup\left(\bigcup_{k,n\in\mathbb{N}}f((v_{k}^n))\right)\cup((0),(1)) \subset T2^{\omega}.
\]
where $f((v_k^n)_{k\in\mathbb{N}})=u_n$. From example \ref{E:0:1} we deduce that 
there exist $\mathfrak{T}_B$ a countable family of infinite paths of the subgraph $T_{B}:=\left(\bigcup_{x\in B}f(x)\right)\cup((0),(1))\subset T_X$ such that $T_B=\bigcup_{\gamma\in\mathfrak{T}_B}\gamma$ and any two different infinite paths belong to $\mathfrak{T}_B$ intersect in at most one vertex. We now construct inductively the rest of the desired family of infinite paths. For the infinite path $f((v_k^1)_{k\in\mathbb{N}}):=((v_1^1),(v_1^1,v_2^1),...)=u_1$ in $T_X$ there exist $j(1)\in \mathbb{N}$ such that $(v_1^1,\ldots,v_{j(1)}^{1})\in  T_B$ but for each $i>j(1)$,  $(v_1^1,\ldots,v_{i}^{1})\notin  T_B$. We define $\gamma_{1}:=(v_{j(1)}^{1},v_{j(1)+1}^{1},...)$ and the subgraph
\begin{equation}\label{eq:5}
T^{1}_{X}:=T_{B}\cup \{\gamma_{1}\} \subset T_X.
\end{equation}
Remark that  $\mathfrak{T}_X^1:=\mathfrak{T}_B\cup\{\gamma_1\}$ is a countable family of infinite paths in $T_X^1$, such that $T^{1}_{X}=\bigcup_{\gamma\in \mathfrak{T}_X^1}\gamma$, and any two different paths in $\mathfrak{T}_X^1$ intersect in at most one vertex. Now suppose we have found the desired countable family $\mathfrak{T}_{X}^{n-1}$ of infinite paths
for the subgraph $T_{X}^{n-1}\subset T_X$ and let $f((v_k^n)_{k\in\mathbb{N}}):=((v_1^n),(v_1^n,v_2^n),...)=u_n$ in $T_X$. Then there exist $j(n)\in \mathbb{N}$ such that $(v_1^n,\ldots,v_{j(n)}^{n})\in  T_B$ but for each $i>j(n)$,  $(v_1^n,\ldots,v_{i}^{n})\notin  T_B$. We define $\gamma_{n}:=(v_{j(n)}^{n},v_{j(n)+1}^{n},...)$ and the subgraph
\begin{equation}\label{eq:6}
T^{n}_{X}:=T_{X}^{n-1}\cup \{\gamma_{n} \}\subset T_X.
\end{equation}
The desired countable family of infinite paths is given by 
$\mathfrak{T}_X:=\mathfrak{T}_B \cup\{\gamma_n:n\in\mathbb{N}\}$.
We finish the proof by remarking that if $X$ is countable the desired family of countable infinite paths is obtained from the preceding recursive construction taking $T_X^1=\{\gamma_1\}$ as base case.
\end{proof}

\begin{corolario}\label{c:1.3}
Let $X$ be a countable closed subset of the Cantor set. Then the sets $Ends(T_X)$ and $\mathfrak{T}_X$ are in bijection.
\end{corolario}

By the way we constructed the family $\mathfrak{T}_X$ (see remark \ref{REM:InfPathsGraph}) we obtain the following:

\begin{corolario}\label{c:1.4}
Let $X=B\sqcup U$ be the decomposition of an uncountable subset of the Cantor set given by theorem \ref{t:1.2}. Suppose that the boundary $\partial U\subset B$ is just one point. Then there exist an infinite path $\widetilde{\gamma}\in\mathfrak{T}_X$ such that the end $[V_n]_{n\in\mathbb{N}}$ of $T_X$ defined by $\widetilde{\gamma}$ is precisely  $\partial U$.
\end{corolario}


\subsection{Translation surfaces and the Veech group}
\label{SS:TSVG}
 We recall some general aspects of translation surfaces. For more details we refer the reader to \cite{PSV} and references within.  A translation surface is a real surface $S$ whose transition functions are translations. Every translation surface inherits a natural flat metric from the plane via pull back. We denote by $\widehat{S}$ the \textit{metric completion} of $S$ with respect to its natural flat metric.


\begin{definicion}\cite{PSV}
\label{d:1.5}
A translation surface $S$ is called \emph{tame} if for every point $x\in \widehat{S}$ there exist a neighborhood $U_{x}\subset \widehat{S}$ which is either isometric to some neighborhood of the Euclidean plane or to the neighborhood of the branching point of a cyclic branched covering of the unit disk in the Euclidean plane. In the later case we call $x$ \emph{a cone angle singularity of angle} $2n\pi$ if the cyclic covering is of (finite) order $n\in\mathbb{N}$ and an \emph{infinite cone angle singularity} when the cyclic covering is infinite.  
\end{definicion}

We denote by $\Sing{(S)}\subset\widehat{S}$ the set of cone angle singularities of $S$. A geodesic in $S$ is called \emph{singular} if it has one endpoint in $\Sing{(S)}$ and no singularities in its interior. A singular geodesic having both endpoints in $\Sing{(S)}$ is called a \emph{saddle connection}. To every saddle connection $\gamma$ we can associate two \emph{holonomy vectors} $\{v,-v\}\subset\mathbb{R}^2$ by developing the translation surface structure along $\gamma$. Analogously, one can associate two unit vectors to every singular geodesic, which we will also call holonomy vectors. Two saddle connections or singular geodesics are said to be \emph{parallel} if their corresponding holonomy vectors are parallel.


An \emph{affine diffeomophism} is a map $f:S\to S$ which is affine on charts. We denote by 
$\Aff_+{(S)}$ the group of all orientation-preserving affine diffeomorphisms. Given that $S$ is a translation surface, the differential of every element in $\Aff_+{(S)}$ is constant. Hence we have a well defined group morphism:
$$
D: \Aff_{+}(S)\to {\rm GL_{+}(2,\mathbb{R})},
$$
where $D(f)$ is the differential matrix of $f$. The image of $D$, that we denote by $\Gamma(S):=D(\Aff_{+}(S))$, is called \emph{the Veech group} of $S$, see \cite{Vee}.

We finish this section by recalling that $ {\rm GL_{+}(2,\mathbb{R})}$ acts on the set of all translation surfaces by post composition on charts. For every $g\in  {\rm GL_{+}(2,\mathbb{R})}$, we denote by $S_g:= g\cdot S$ and $\overline{g}:S_{Id}\to S_{g}$ the corresponding affine diffeomorphism.


\subsection{Auxiliary Constructions} In the following paragraphs we introduce some elementary constructions needed to prove our main results. All these constructions are based on the same principle: \emph{to glue translation surfaces along parallel marks}. 
A \textit{mark} $m$ on a translation surface $S$ is finite length geodesic having no singular points in its interior. As with saddle connections, we can associate to each mark 
two holonomy vectors by developing the translation structure along them. Two marks on $S$ are parallel if their respective holonomy vectors are parallel.
\begin{definicion}[Gluing marks]
\label{d:1.7}
Let $m$ and $m^{'}$ be two disjoint parallel marks on a translation surface $S$. We cut $S$ along $m$ and $m'$, which turns $S$ into a surface with boundary consisting of four straight segments. We glue this segments back using translations to obtain a tame translation surface $S'$ \emph{different} from the one we started from. We say that $S'$ is obtained from $S$ by \emph{regluing} along $m$ and $m'$.

\begin{figure}[h!]
 \centering
 \includegraphics[scale=0.6]{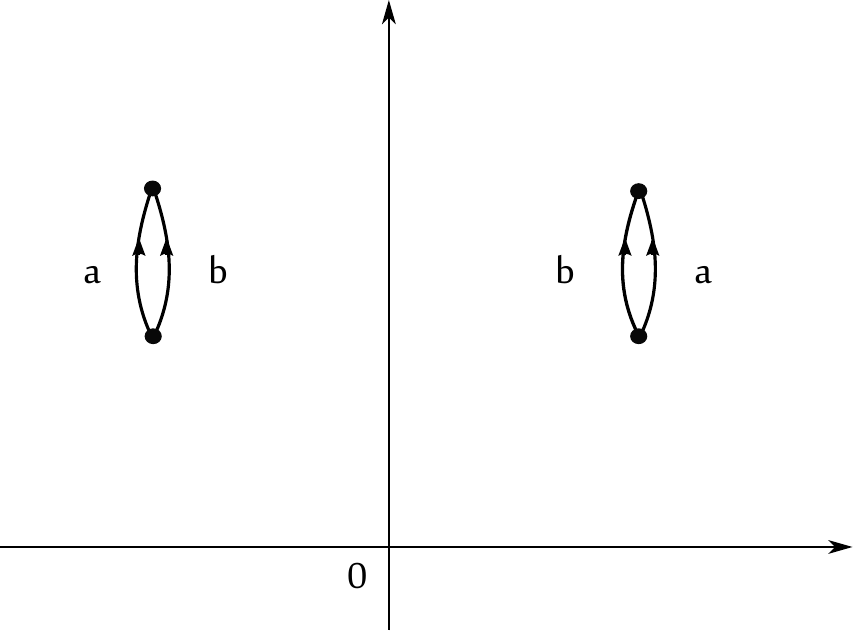}\\
 \caption{\textit{Gluing marks.}}
 \label{Figure5}
 \end{figure}
\end{definicion}

We denote by $m\sim_{glue}m^{'}$ the operation of glueing the marks $m$ and $m'$ and $S'=S / m\sim_{glue}m^{'}$. In Figure \ref{Figure5} we depict the glueing of two marks on the plane. Remark that the operation of glueing marks can also be performed for marks on different surfaces. In any case, $\Sing{S'}\setminus\Sing{S}$ is formed by two $4\pi$ cone angle singularities, that is, $S$ tame implies $S'$ tame.


\begin{lema}\label{l:1.3}
Let $S_{1}$ and $S_2$ be two translation surfaces homeomorphic to the Loch Ness Monster and $M^j:=\{m_i^j:\forall i\in\mathbb{N}\}$ a discrete\footnote{By discrete we mean that $M^j$, as a set of marks, does not accumulate in the metric completion of the surface.} family of marks on $S_j$, $j=1,2$ such that $m^1_i$ and $m^2_i$ are parallel, for every $i\in \mathbb{N}$. Then 
\[
S:=\left(\bigcup\limits_{j\in\{1,2\}}S_j \right)\Big/ m_i^1\sim_{\text{glue}}m^2_i, \text{ for every } i\in\mathbb{N},
\]
is a tame translation surface homeomorphic to the Loch Ness Monster (see Figure \ref{Figure6}).
\begin{figure}[h]
 \centering
 \includegraphics[scale=0.8]{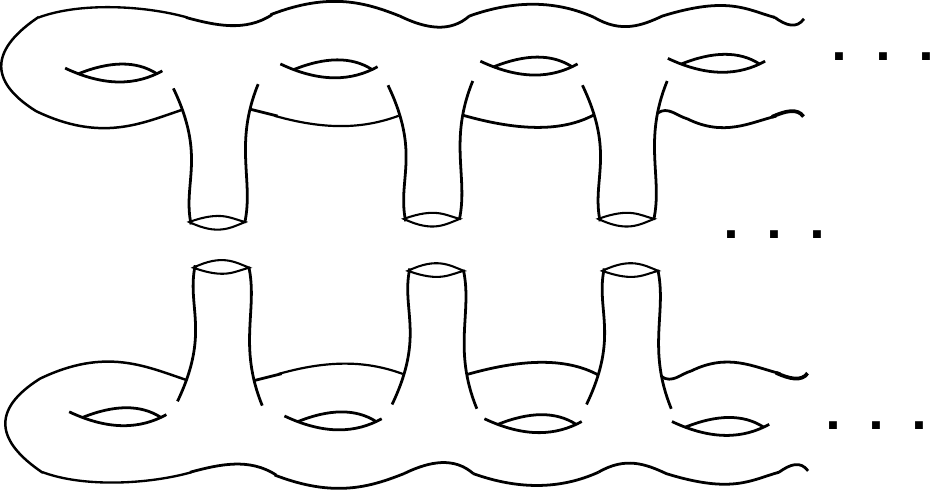}
 \caption{\emph{Gluing marks on the Loch Ness Monsters.}}
 \label{Figure6}
\end{figure}
\end{lema}
\begin{proof}
Given that family of marks $M^j$ does not accumulate in $\widehat{S}_j$ the quotient $S$ is a tame translation surface. Let $K$ be a compact subset of $S$. We will show that there exists a compact subset $K\subset K'\subset S$ such that $S\setminus K'$ is connected. By the remark done after theorem \ref{t:1.1}, this implies that $S$ has only one end. Let $f:S_1\to S_2$ be a homeomorphism such that $f(m^1_{i})=m^2_i$, for each $i\in\mathbb{N}$ and define the projections
 \[
\begin{array}{ccccc}
\pi_1 & : &  \bigcup\limits_{j\in\{1,2\}}(S_j\setminus M^j) & \to & (S_1\setminus M^{1})\\
         &   &             x                                                             & \to & \begin{cases}
                                                                                                         x, & \text{if $x\in S_1$,} \\
                                                                                                         f^{-1}(x), & \text{if $x\in S_2$}.
                                                                                                          \end{cases}
\end{array}
\]
\[
\begin{array}{ccccc}
\pi_2 & : &  \bigcup\limits_{j\in\{1,2\}}(S_j\setminus M^j) & \to & (S_2\setminus M^{2})\\
         &   &             x                                                             & \to & \begin{cases}
                                                                                                         x, & \text{if $x\in S_2$,} \\
                                                                                                         f(x), & \text{if $x\in S_1$}.
                                                                                                          \end{cases}
\end{array}
\]
We denote by $\pi:S_1\cup S_2\to S$ the standard projection. For each $j\in\{1,2\}$, the closure of $\pi_j\left(\pi^{-1}(K)\setminus (M^1\cup M^2)\right)$ in $S_j$ is compact. We denote it by $K_j$. Since $S_j$ has only one end, there exist a compact subset $K^{'}_j\subset S_j$ such that $K_j\subset K^{'}_j$ and $S_j\setminus K^{'}_j$ is connected. Define $K'$ as the closure of
\[
\pi\left(\bigcup\limits_{j\in\{1,2\}} \pi^{-1}_{j}\left(K^{'}_j\cap (S_j\setminus M^j)\right)\right)
\]
in $S$. By construction, $K\subset K^{'}$ and $S\setminus K^{'}$ is connected. To conclude the proof remark that, since cutting along marks does not destroy genus, for each $j\in\{1,2\}$ the open surface $S_j\setminus M^j$ with boundary of $S_j$ has infinite genus and is naturally embedded in $S$. 
\end{proof}

\begin{construccion}[\emph{Elementary piece}]
\label{cons:1.1}
Let $X$ be a closed subset of the Cantor set, $T_X$ the subgraph of $T2^{\omega}$ given by Lemma \ref{l:1.1} and $\mathfrak{T}_X$ a countable family of infinite paths as in Lemma \ref{l:1.2}. Suppose that for each path $\gamma\in \mathfrak{T}_X$ we have a tame translation surface $S(\gamma)$ homeomorphic to the Loch Ness Monster and a countable family of marks $M:=\{m_k: k \in\mathbb{N}\}\subset S(\gamma)$ which do not accumulate on the metric completion $\widehat{S(\gamma)}$. Suppose that the vertex set of $\gamma$ is labeled by $(v_k=(x_1,\ldots,x_k))$ as in (\ref{eq:4}). For each $k\in\mathbb{N}$ we label the $k$-th mark $m_k$ of the family $M$ with $v_k\in\gamma$ and define the \emph{elementary piece associated to $X$ and the family of translation surfaces $\{S(\gamma)\}_{\gamma\in\mathfrak{T}_X}$} as : 
\begin{equation}\label{eq:8}
S_{elem}:= \bigcup\limits_{\gamma\in\mathfrak{T}_X}S(\gamma) \Big/ \sim
\end{equation}
where $\sim$ is the equivalent relation given by glueing marks with the same labels.
\end{construccion}

\begin{lema}\label{l:1.4}
For every closed subset $X$ of the Cantor set, $S_{elem}=S_{elem}(X)$ is a tame translation surface with $Ends(S_{elem})=Ends_\infty(S_{elem})$  homeomorphic to $X$.
\end{lema}
\begin{proof} 
Tameness follows from two facts: any two infinite paths in $\mathfrak{T}_X$ intersect in at most one vertex and marks never accumulate on the metric completion $\widehat{S(\gamma)}$. The rest of the proof goes as follows. First we prove that $Ends(S_{elem})$ is homeomorphic to $X$. The idea is to define an embedding $i:T_X\hookrightarrow S_{elem}$ and show that it induces a homeomorphisms between $Ends(T_X)$ and $Ends(S_{elem})$. Finally, we prove that $S_{elem}$ has no planar ends.

\emph{The embedding.} Let $[U(\gamma)_n]_{n\in\mathbb{N}}$ be the end of $S(\gamma)$. Up to making an isotopy on $S(\gamma)$, for every $k\in\mathbb{N}$ we can suppose that the marks $\{v_1,\ldots,v_k\}\subset S(\gamma)$ are contained in $S(\gamma)\setminus U(\gamma)_k$. It is then easy to see that there exists an infinite \emph{simple} embedded path $\delta:[0,\infty)\to S(\gamma)$ such that $\delta(0)$ is an extremity of the mark labeled with the vertex $v_1$, all marks $M$ are contained in $\delta([0,\infty])$ and $\delta^{-1}(U(\gamma)_k)=(n,\infty)$. In other words, the image of $\delta$ goes only once over each mark\footnote{Except for the mark labeled with the first vertex $v_1$.} and runs into the only end of the surface. On the vertex set $V(\gamma)=(v_k)_{k\in\mathbb{N}}$, we define $i_{\gamma}(v_k)$ as the only endpoint in the mark labeled with $v_k$ satisfying the following property: if $\overline{i_{\gamma}(v_k)i_{\gamma}(v_{k+1})}$ is the segment in the image of $\delta$ going from $i_{\gamma}(v_k)$ to $i_{\gamma}(v_{k+1})$ and $p_{k+1}$ is the endpoint on the mark $v_{k+1}$ different from $i_{\gamma}(v_{k+1})$, then $p_{k+1}\in\overline{i_{\gamma}(v_k)i_{\gamma}(v_{k+1})}$. We extend the map $i_{\gamma}$ to all edges by declaring  the image of the edge $(v_k,v_{k+1})$ to be 
$\overline{i_{\gamma}(v_k)i_{\gamma}(v_{k+1})}$. This defines an embedding:
\[
i_{\gamma}:\gamma\hookrightarrow S(\gamma),
\]
by declaring that the image of the edge $(v_k,v_{k+1})$ is precisely the subarc of $\delta$ joining $i_{\gamma}(v_k)$ to $i_\gamma(v_{k+1})$. Now consider two infinite paths $\gamma,\gamma'\in\mathfrak{T}_X$ sharing a vertex $v_k$. This means that $S(\gamma)$ and $S(\gamma')$ are glued along the mark labeled with $v_k$ when constructing $S_{elem}$. By the way we defined $i_\gamma$ and $i_{\gamma'}$ we have that their images in $S_{elem}$ only intersect in $i_\gamma(v_k)=i_{\gamma'}(v_k)$. Therefore we can glue the family of maps $\{i_\gamma\}_{\gamma\in\mathfrak{T}_X}$ to define an embedding:
\begin{equation}\label{eq:9}
i:T_{X}\hookrightarrow S_{elem},
\end{equation}
Remark that in general we cannot use the set of paths in $\mathfrak{T}_X$ to characterise \emph{all} points in $Ends(T_X)$ because $\mathfrak{T}_X$ is always countable but $Ends(\mathfrak{T}_X)$'s cardinality is the same as $X$'s, which may be uncountable. To overcome this difficulty, let us define the family $\mathfrak{D}_X$ of \emph{descending paths} of $T_X$ as the set of all infinite paths $(D_{n})_{n\in \mathbb{N}}$ of $T_X$ such that $D_n\in 2^n$ (see \S\ref{SS:CantorBinaryTree}). Every descendent path $(D_n)_{n\in \mathbb{N}}$ defines an element $[\mathcal{U}_n]_{n\in \mathbb{N}}\in Ends(T_X)$ as follows: for every vertex $D_n$ let $\mathcal{U}_n\subset T_X$ be the connected componen of $T_X\setminus D_n$ containing $D_{n+1}$. This defines a bijection between $Ends(T_X)$ and $\mathfrak{D}_X$. We use the family of descending paths of $T_X$ to define a homeomorphism between $Ends(T_X)$ and $Ends(S_{elem})$.

Let $(D_n)_{n\in \mathbb{N}}\in\mathfrak{D}_X$ be the descending path determining $[\mathcal{U}_n]_{n\in \mathbb{N}}\in Ends(T_X)$. For each vertex $D_n$ we consider two cases:
	\begin{enumerate}
	\item Suppose that there exist two paths $\gamma_\alpha$, $\gamma_\beta\in\mathfrak{T}_X$ such that $D_n\in\gamma_\beta\cap\gamma_\alpha$. Let $k_0,k_0'\in\mathbb{N}$ be the smallest positive integers such that $i_{\gamma_\alpha}(D_n)\not\in U(\gamma_\alpha)_{k_0}\in[U(\gamma_\alpha)_k]_{k\in\mathbb{N}}$ and $i_{\gamma_\beta}(D_n)\not\in U(\gamma_\beta)_{k'_0}\in[U(\gamma_\beta)_k]_{k\in\mathbb{N}}$. We define then $W_n$ as the connected component of $S_{elem}\setminus \partial U(\gamma_\alpha)_{k_0} \cup \partial U(\gamma_\beta)_{k_0'} $ containing $i(D_{n+1})$.
\item Suppose $D_n\in\gamma_\alpha\in\mathfrak{T}_X$ but is not contained in any other infinite path of $\mathfrak{T}_X$. Let $k_0$ be the smallest positive integer such that 
$i_{\gamma_\alpha}(D_n)\not\in U(\gamma_\alpha)_{k_0}\in[U(\gamma_\alpha)_k]_{k\in\mathbb{N}}$. We define then $W_n$ as the connected component of $S_{elem}\setminus \partial U(\gamma_\alpha)_{k_0} $ containing $i(D_{n+1})$.
\end{enumerate}
It is easy to check that the map 
$$
i_*:Ends(T_X)\to Ends(S_{elem})
$$
given by $i_*([\mathcal{U}_n]_{n\in \mathbb{N}}):=[W_n]_{n\in \mathbb{N}}$ is well defined. We now prove that it is a closed continuous bijection, hence a homeomorphism.

\emph{Injectivity}. Consider two different infinite paths $(D_n)_{n\in \mathbb{N}}$ and $(D_n')_{n\in \mathbb{N}}$ in $\mathfrak{D}_X$ defining $[\mathcal{U}_n]_{n\in \mathbb{N}}$ and $[\mathcal{U}'_n]_{n\in \mathbb{N}}$ ends of $T_X$. Then there exists $N\in\mathbb{N}$ such that for all $m>N$ we have that $D_m\neq D'_m$. This implies that $W_m\cap W'_m=\emptyset$, hence $i_*([\mathcal{U}_n]_{n\in \mathbb{N}})=[W_n]_{n\in \mathbb{N}}\neq [W_n']_{n\in \mathbb{N}}=i_*([\mathcal{U'}_n]_{n\in \mathbb{N}})$.

\emph{Surjectivity}. Consider an end $[W_{n}]_{n\in \mathbb{N}}$ of $S_{elem}$. By the way the embedding (\ref{eq:9}) was defined,  $[i^{-1}(W_{n}\cap i(T_X))]$ defines an end in $T_X$. Let $(D_n)_{n\in \mathbb{N}}\in\mathfrak{D}_X$ be the descending path defining $[i^{-1}(W_{n}\cap i(T_X))]$ and $[\mathcal{U}_n]_{n\in \mathbb{N}}=[i^{-1}(W_{n}\cap i(T_X))]$. Then $i_*([\mathcal{U}_n]_{n\in \mathbb{N}})=[W_{n}]_{n\in\mathbb{N}}$.

\emph{The map $i_{\ast}$ is continuous.} Let $W\subset S_{elem}$ be an open set with compact boundary and $W^*\subset Ends(S_{elem})$ the basic open set it defines in the space of ends. Since (\ref{eq:9}) is an embbeding, we have that $\mathcal{U}:=i^{-1}(W\cap T_X)$ is an open set  with compact boundary and $i_*(\mathcal{U}^*)\subset W^*$. 

The map $i_*$ is closed for it is a bijection from a compact Hausdorff space into a Hausdorff space. Finally we remark that $S_{elem}$ has no planar ends for each piece $S(\gamma)$ used in its construction has infinite genus. 
\end{proof}

\section{Proof Theorem \ref{T:PP}}
	\label{SEC:UNCOUNTABLEVEECH}

The proof that we present relies on construction \ref{cons:1.1} and lemma \ref{l:1.4}. First we define a tame translation surface $S_P$ homeomorphic to the Loch Ness Monster whose Veech group is exactly $P$. This surface comes with an infinite family of marks which do not accumulate on the boundary, hence we are in shape to perform construction \ref{cons:1.1} taking all $S(\gamma)$ to be equal to $S_P$ and $X$ an arbitrary closed subset of the Cantor set. We then check that the Veech group of the resulting elementary piece $S_{elem}$ is $P$. The case for $P'$ is treated analogously.

\begin{construccion}
\label{cons:2.1}
Consider $\mathbb{E}$ a copy of the Euclidean plane equipped with a fixed origin $\overline{0}$ and an orthogonal basis $\beta= \{e_{1},e_{2}\}$. On $\mathbb{E}$ we 
define\footnote{Marks are given by their ends points.} two infinite families of marks:
\[
L:=  \{l_{i} =((4i-1)e_{1}, \, 4ie_{1}) : \forall i \in \mathbb{N}\} \, \,\text{ and, }\, \,
M:= \{m_{i} =((4i-3)e_{1}, \, (4i-2)e_{1}) : \forall i \in \mathbb{N} \}.
\]
and the tame Loch Ness Monster (see Figure \ref{Figure8}):
\begin{equation}\label{eq:11}
S_P:=\mathbb{E}\big/ l_{2i-1}\sim_{\text{glue}} l_{2i},  \text{ for every } i \in \mathbb{N}
\end{equation}
\begin{figure}[h!]
  \centering
  \includegraphics[scale=0.9]{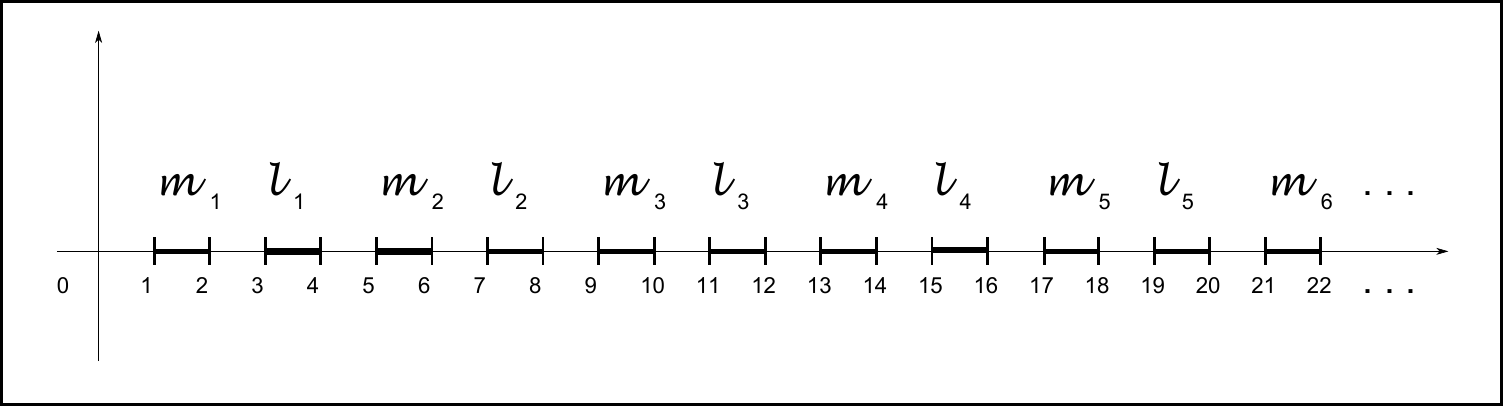}
  \caption{\emph{Tame Loch Ness Monster $S_P$.}}
  \label{Figure8}
\end{figure}
\end{construccion}

Given that $\{\pm e_1\}$ is the set of all holonomy vectors of $S_P$ and there are no saddle connections on the half plane $x<0$, the Veech group of $S_P$ is $P$. We remark that $S_P$ comes with the infinite family of marks $M=\{m_k\}_{k\in\mathbb{N}}$, none of which has been glued to another mark.  

Let $X$ be a closed subset of the Cantor set and 
\begin{equation}\label{eq:12}
S_{elem}:= \bigcup\limits_{\gamma\in\mathfrak{T}_X}S(\gamma) \Big/ \sim,
\end{equation}
be the tame translation surface obtained by performing construction \ref{cons:1.1} with initial data $X$ and all $S(\gamma)$ equal to $S_P$. By lemma \ref{l:1.4}, $S_{elem}$ has no planar ends and its ends space is homeomorphic to $X$. Since all marks $M$ in $S_P$ are parallel, all saddle connections in $S_{elem}$ are parallel as well. Hence we can define for each $g\in P$ an affine diffeomorphism $f$ on $S_P$  whose differential is exactly $g$ and which \emph{fixes} all points in the marks belonging to $L$ and $M$. Since $S_{elem}$ is obtained by glueing copies of $S_P$ along saddle connections, these affine diffeomorphisms can be glued together to define a global affine diffeomorphism $f\in \Aff_+{(S_{elem})}$ whose differential is precisely $g$. In other words $P$ is a subgroup of $\Gamma(S_{elem})$. 
To see that $\Gamma(S_{elem})$ is a subgroup of $P$ it is sufficient to remark that
there is only one horizontal direction defining infinite singular geodesics in $S_{elem}$ and this direction has to be fixed by the (linear) action of $\Gamma(S_{elem})$ on the plane.
The construction of a tame translation surface having ends space homeomorphic to $X$ and Veech group equal to $P'$ done exactly as we did for $P$, except that instead of taking all copies of $S(\gamma)$ equal to $S_P$ we take them equal to the surface $S_{P'}$ defined in the following paragraph. 
\begin{construccion}
\label{cons:2.2}
Consider $\mathbb{E}$ a copy of the Euclidean plane equipped with a fixed origin $\overline{0}$ and an orthogonal basis $\beta= \{e_{1},e_{2}\}$. On $\mathbb{E}$ we 
define two infinite families of marks:
\[
\begin{array}{lcl}
M^{+} := \{ m^{+}_{i} = ((4i-3)e_{1}, \, (4i-2)e_{1}) : \forall i \in \mathbb{N}\}, & &L^{+} :=  \{ l^{+}_{i}  =   ((4i-1)e_{1}, \, 4ie_{1}) : \forall i \in \mathbb{N}\},\\
 M^{-}:= \{ m^{-}_{i}   = ((3-4i)e_{1}, \, (2-4i)e_{1}) : \forall i \in \mathbb{N}\} &\text{and}& L^{-}:=\{ l^{-}_{i} = ((1-4i)e_{1}, \, -4ie_{1}): \forall i \in \mathbb{N}\}.
\end{array}
\]
and the tame Loch Ness Monster (see Figure \ref{Figure10}).
\begin{equation}\label{eq:13}
S_{P^{'}}:=\mathbb{E}\big/l^{+}_{2i-1}\sim_{\text{glue}} l^{+}_{2i} \text{ and } l^{-}_{2i-1}\sim_{\text{glue}} l^{-}_{2i}, \text{  for each }i \in \mathbb{N}
\end{equation}
\qed
\begin{figure}[h!]
   \centering
   \includegraphics[scale=0.8]{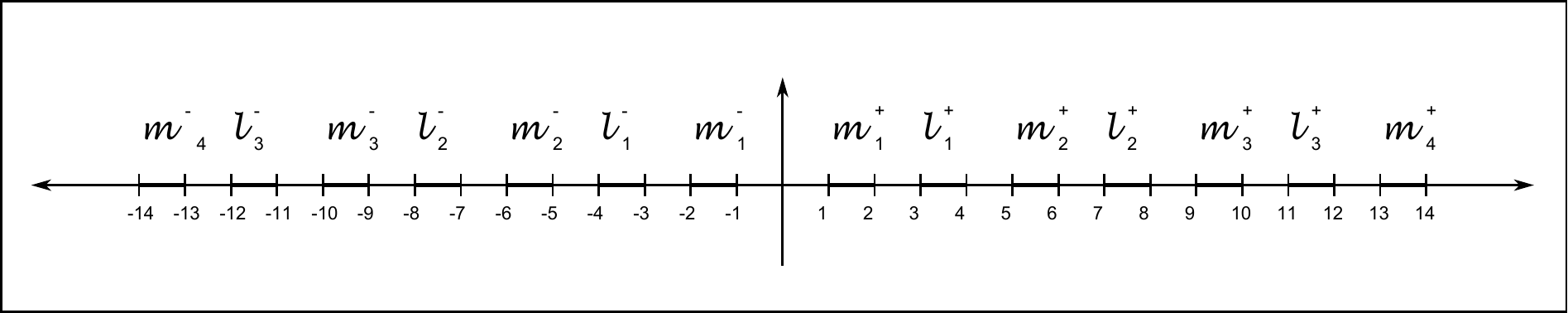}
   \caption{\emph{Tame Loch Ness Monster $S_{P^{'}}$.}}
   \label{Figure10}
\end{figure}
\end{construccion}

\begin{corolario}
\label{c:2.1}
\begin{upshape}
There exist a tame translation surface $S$ of genus zero with Veech group $P$ such that its space of ends is homeomorphic to $X$. \end{upshape}
\end{corolario}

Indeed, consider $\mathbb{E}$ the copy of the Euclidean plane described by the Construction \ref{cons:2.1} with only the family of marks $M$ and proceed verbatim as in case 1.

\section{Proof of main results}
	\label{Section:PUZZLES}

In this chapter we present the proofs of theorems \ref{T:AF}, \ref{T:AN} and \ref{T:NN}. These follow from a general construction that we introduce in the first section of this chapter.

\subsection{Puzzles} Let $(X,G,H)$ be a triple where $X$ is a closed subset of the Cantor set, $G$ is a countable subgroup of ${\rm GL_+(2,\mathbb{R})}$ without contracting elements generated\footnote{We think of $H$ as generator subset closed under inverse elements.} by $H$. To the triple $(X,G,H)$ we will associate a family of tame translation surfaces $\mathfrak{P}(X,G,H)=\{S_g: g\in G\}$. This set will be called \emph{a puzzle} and each of its elements \emph{a piece}. As notation suggests, each piece of the puzzle is an affine copy of a fixed tame translation surface, which will be called \emph{the elementary piece of the puzzle} for it is obtained through construction 
\ref{cons:1.1}. The pieces $S_g$ of the puzzle are endowed with families of marks that do not accumulate in the metric completion, hence we can glue them together to form a \emph{tame} translation surface $S_{\mathfrak{P}}$ that we will call \emph{the assembled surface}. 
The Veech group of the surface $S_{\mathfrak{P}}$ is $G$. Moreover, this surface satisfies $Ends(S_\mathfrak{P})=Ends_{\infty}(S_\mathfrak{P})$,  and we will give an explicit description of this space of ends that will allow us to prove, for different instances of $X$, theorems \ref{T:AF}, \ref{T:AN} and \ref{T:NN}.

\textbf{The elementary piece of a puzzle}. As it name suggest, this piece is obtained using construction \ref{cons:1.1}. Recall that this construction has as initial data a closed subset $X$ of the Cantor set and, for each path $\gamma$ in the countable family of paths $\mathfrak{T}_X$ given by lemma \ref{l:1.2}, a tame Loch Ness Monster $S(\gamma)$ endowed with an infinite family of marks that do not accumulate. We build the surfaces $S(\gamma)$ glueing together 
a series of tame translation surfaces that we define in what follows. With this purpose in mind, we choose an enumeration\footnote{By enumeration of these groups when $|G|$ or $|H|$ is infinite we mean to write them as an infinite sequence.} $G:=\{g_{1},...,g_{|G|}\}$ and $H:=\{h_{1},...,h_{|H|}\}$ for elements in $G$ and $H$ respectively.

\emph{Buffer Loch Ness Monster}. For every $h_j\in H$ and $g\in G$ we construct a tame translation surface $S(g,gh_j)$ homeomorphic to the Loch Ness Monster. The purpose of these surfaces is to separate singularities during the construction of $S_{\mathfrak{P}}$, guaranteeing thus tameness. This idea already appears in construction 4.4 in \cite{PSV}. 

For each element $h_{j}\in H$, consider $\mathbb{E}(j,1)$ and $\mathbb{E}(j,2)$ two copies of the Euclidean plane equipped with origins $\overline{0}$ and an orthogonal basis $\beta= \{e_{1},e_{2}\}$. On $\mathbb{E}(j,1)$ we draw the following two families of marks:
\[
\check{M}^{j} := \{ \check{m}^{j}_{i} = (4ie_{1}, (4i+1)e_{1}) :\forall i \in \mathbb{N}\} \text{ and }
 L_1:= \{ l_{i}^1 = ((4i+2)e_{1}, \, (4i+3)e_{1}):\forall i \in \mathbb{N}\},
 \]
and on $\mathbb{E}(j,2)$:
\[
h_{j}\check{M}^{-j} := \{ h_{j}\check{m}^{-j}_{i}  =  (2ie_{2}, \, e_{1}+2ie_{2}) : \forall i \in \mathbb{N}\} \text{ and }
L_2 := \{l^{2}_{i} = ((2i+1)e_{2}, \, e_{1}+(2i+1)e_{2}) :\forall i \in \mathbb{N}\}.
\]
We define (see Figure \ref{Figure11}) \emph{the buffer surface}
\begin{equation}\label{eq:15}
S(Id,h_{j}):=\bigcup_{k=1}^{2}\mathbb{E}(j,k)\bigg/l^1_{i}\sim_{\text{glue}} l^{2}_{i}, \text{ for each } i\in \mathbb{N}.
\end{equation}
\begin{figure}[h!]
  \centering
  \includegraphics[scale=0.45]{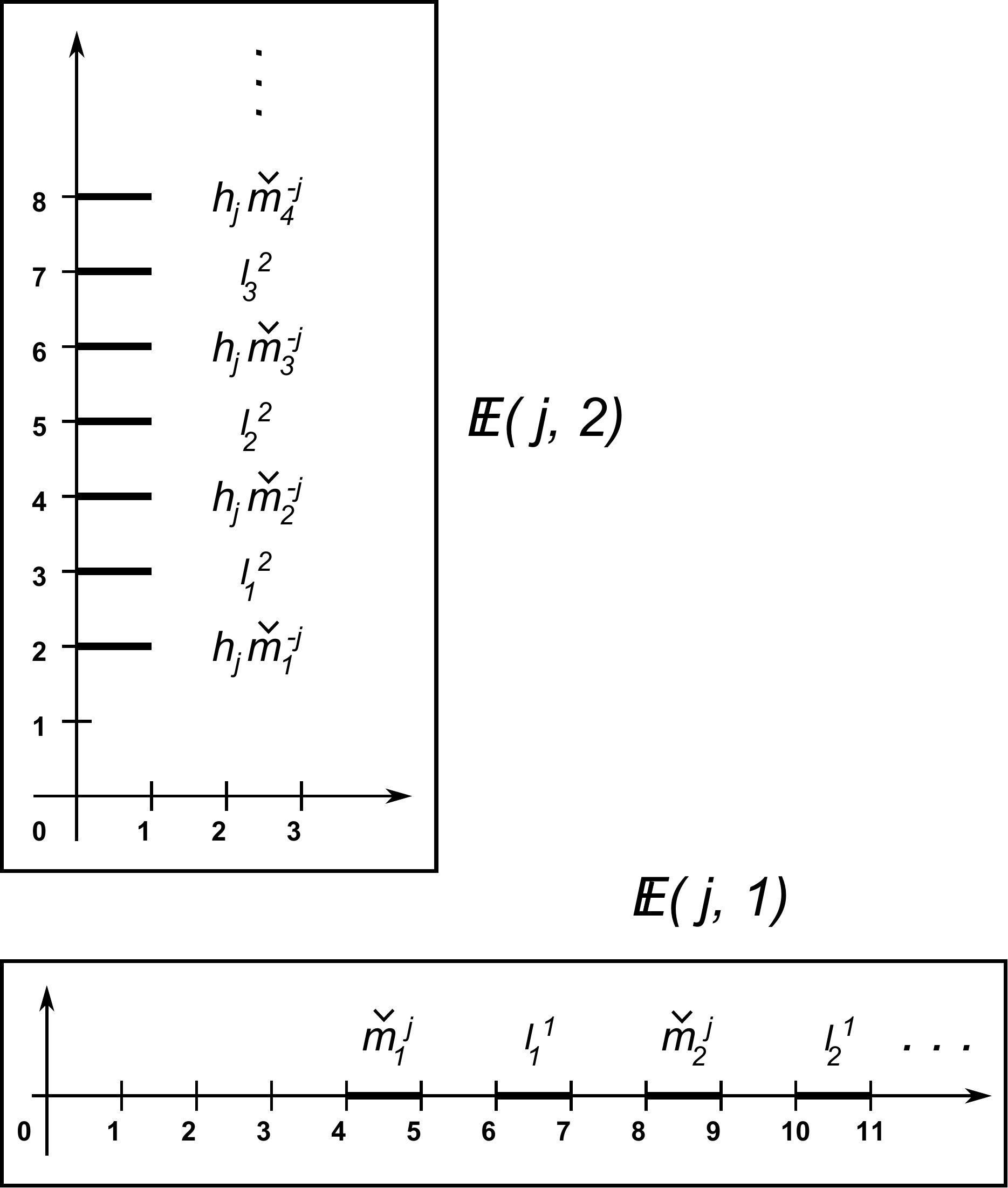}
  \caption{\emph{Buffer Loch Ness Monster $S(Id,h_j)$.}}
  \label{Figure11}
\end{figure}

Remark that by Lemma \ref{l:1.3} the surface $S(Id,h_j)$ is a tame translation surface with infinitely many conic singularities of angle $4\pi$ homeomorphic to the Loch Ness Monster. For each $g\in G$, we define:
\begin{equation}
S(g,gh_j):=g\cdot S(Id,h_j)
\end{equation}
That is, $S(g,gh_j)$ is the affine copy of $S(Id,h_j)$ obtained by postcomposing every chart by the affine transformation associated to the matrix $g$. The family of marks $\check{M}^{j}$ and $h_{j}\check{M}^{-j}$ on $S(g,gh_j)$ are relabeled as $g\check{M}^{j}$ and $gh_{j}\check{M}^{-j}$, respectively. These marks will be used later to construct $S_{\mathfrak{P}}$. As said before, the purpose of the surfaces $S(g,gh_j)$ is to separate singularities during the construction of $S_{\mathfrak{P}}$, guaranteeing thus tameness. The following lemma is essential to assure this property.
\begin{lema}
\label{l:3.1}
\cite[Lemma 4.5]{PSV} For every $g\in G$ and $h_j\in H$ the distance between $g\check{M}^{j}$ and $gh_{j}\check{M}^{-j}$ is at least $\frac{1}{\sqrt{2}}$.
\end{lema}

\emph{Decorated Loch Ness Monster}. Using $S(Id,h_j)$ defined above, we construct a tame translation surface $S$ homeomorphic to the Loch Ness Monster and having only one conic singularity of angle $6\pi$. This surface will be called the decorated Loch Ness Monster and its purpose is to force the Veech group of $S_{\mathfrak{P}}$ to be exactly $G$. This idea also already appears in construction 4.6 in \cite{PSV}.

Consider $\mathbb{E}$ a copy of the Euclidean plane equipped with an origin $\overline{0}$ and the same orthogonal basis $\beta= \{e_{1},e_{2}\}$ as before. On $\mathbb{E}$ we draw the following families or marks: 
\begin{equation}\label{eq:16}
\begin{array}{cl}
  M     & :=\{m_{i} = ((4i-1)e_{1}, \, 4ie_{1}): \forall i\in \mathbb{N}\} \text{ and} \\
  M^{j} & :=\{m^{j}_{i} = ((2i-1)e_{1}+(j+1)e_{2}, \, 2ie_{1}+(j+1)e_{2}): \forall  i \in \mathbb{N}, \, \forall j\in \{0,...,|H|\}.
\end{array}
\end{equation}
Now, we shall define recursively new families of marks on $\mathbb{E}$.

For $j=1$. We can choose a point $(x_{1},y_{1}) \in \mathbb{E}$ where $x_1>0$ and $y_1<0$ such that the family of marks
\[
M^{-1}:=\{m^{-1}_{i} = (ix_{1}e_{1}+y_{1}e_{2}, \, ix_{1}e_1+h^{-1}_{1}e_{1}+y_{1}e_{2}): \forall i \in \mathbb{N}\} \subset \mathbb{E},
\]
is disjoint to all marks defined on (\ref{eq:16}). For $|H|\geq j>1$. We can chose a point $(x_{j},y_{j})\in \mathbb{E}$ where $x_j>0$ and $y_j<0$ such that the family of marks
\[
M^{-j}:=\{m^{-j}_{i}=(ix_{j}e_{1}+y_{j}e_{2}, \, ix_{j}e_1+h^{-1}_{j}e_{1}+y_{j}e_{2}): \forall i \in \mathbb{N}\}\subset \mathbb{E},
\]
is disjoint to all marks defined in (\ref{eq:16}) and the step $j-1$.

On the other hand, let  $\pi:\widetilde{\mathbb{E}}\to \mathbb{E}$ be the threefold cyclic covering of  $\mathbb{E}$ branched over the origin and
\[
\widetilde{M}^{0}:=\{\widetilde{m_{i}}^{0}: \forall i \in \mathbb{N}\}.
\]
one of the three (disjoint) families of marks on $\widetilde{\mathbb{E}}$ defined by $\pi^{-1}(M)$. In addition, consider preimages $\widetilde{t_{1}}$ and $\widetilde{t_{2}}$ of $t_{1}:= (e_{2}, \, 2e_{2})$ and $t_{2} := (-e_{2},\, -2e_{2})\in \mathbb{E}$ in $\widetilde{\mathbb{E}}$ respecitvely which are in the same fold of $\widetilde{\mathbb{E}}$ as $\widetilde{M}^{0}$. We define the \emph{decorated surface} as:
\begin{equation}\label{eq:17}
S:= \left(\mathbb{E}\cup \widetilde{\mathbb{E}}\bigcup\limits_{\forall h_{j}\in H}S(Id,h_{j})\right) \Bigg/  \sim,
 \end{equation}
where $\sim$ is the equivalent relation given by glueing marks as follows (see Figure \ref{Figure12}):
\begin{enumerate}
\item $\widetilde{t_{1}}\sim_{\text{glue}}\widetilde{t_{2}}$ on $\widetilde{\mathbb{E}}$.
\item $m^{0}_{i}\sim_{\text{glue}} \widetilde{m_{i}}^{0}$ on $\mathbb{E}$ and $\widetilde{\mathbb{E}}$, respectively.
\item $m_{i}^{j}\sim_{\text{glue}}\check{m}_{i}^{j}$, for each $i\in\mathbb{N}$ and for each $j\in\{1,...,|H|\}$, on $\mathbb{E}$ and $S(Id, h_j)$, respectively.
\end{enumerate}
\begin{figure}[h!]
  \centering
  \includegraphics[scale=0.3]{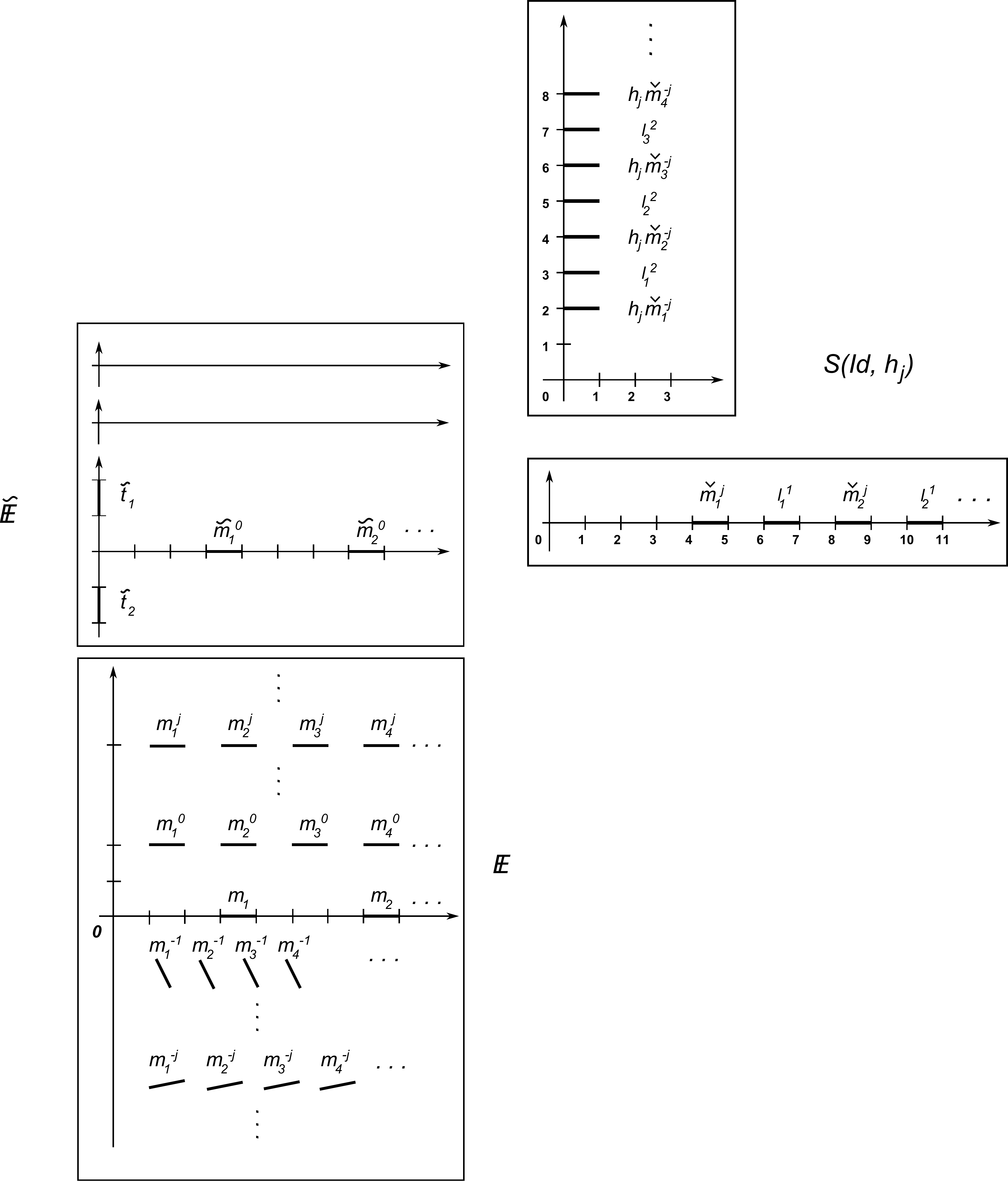}
  \caption{\emph{Decorated Loch Ness Monster $S$.}}
  \label{Figure12}
\end{figure}

By the Lemma \ref{l:1.3} the decorated surface $S$ is a tame translation surface homeomorphic to the Loch Ness Monster with infinitely many cone angle singularities of angle $4\pi$ and only one singular point of angle $6\pi$. 

\begin{remark}
The decorated surface $S$ has $2|H|+1$ families of marks left without gluing. These are $h_j\check{M}^{-j}$, $M^{-j}$ and $M$, where $j\in\{1,\ldots,|H|\}$. These marks will be used in what follows to define the elementary piece of the puzzle and to glue the surfaces $S_g$ forming the puzzle to form the surface $S_{\mathfrak{P}}$.
\end{remark}

\begin{construccion}[\emph{Elementary piece associated to puzzle $\mathfrak{P}(X,G,H)$}]
\label{cons:3.3}
Consider $X$ a closed subset of the Cantor set, the graph $T_{X}$ with ends space homeomorphic to $X$ given by Lemma \ref{l:1.1} and $\mathfrak{T}_X$ the countable family of paths given by lemma \ref{l:1.2} and decomposing $T_{X}$. Pick an arbitrary path $\tilde{\gamma}\in\mathfrak{T}_X$. We define $S(\tilde{\gamma})$ as a copy of the decorated Loch Ness Monster described above. We endow this tame translation surface with the family of marks $M$ (see preceding remark). Clearly, the family $M$ does not accumulate in the metric completion of $S(\tilde{\gamma})$. Moreover, for each $k\in\mathbb{N}$ we label the $k$-th mark $m_k$ of the family $M$ in $S(\tilde{\gamma})$ with the vertex $v_k\in\tilde{\gamma}$.

For each infinite path  $\gamma$ in $\mathfrak{T}_X\setminus\{\tilde{\gamma}\}$ we define $S(\gamma)$ as a copy of the Loch Ness Monster described in the Construction \ref{cons:2.1}. By construction, each $S(\gamma)$ is endowed with a countable family of marks which, abusing notation, we also denote by $M$. These marks do not accumulate on the metric completion. For each $k\in\mathbb{N}$ we label the $k$-th mark $m_k$ of the family $M$ in $S(\gamma)$ with the vertex $v_k\in\gamma$. The \emph{elementary piece associated to the puzzle} $\mathfrak{P}(X,G,H)$, which we denote by $S_{elem}$ is defined as the tame translation surface obtained by performing construction \ref{cons:1.1} on the initial data $X$ and the marked translations surfaces:
$$
S(\tilde{\gamma})\cup\left( \bigcup\limits_{\gamma\in\mathfrak{T}_X\setminus\tilde{\gamma}}S(\gamma)\right)
$$
\end{construccion}

\begin{remark}
\label{rem:3.1}
The elementary piece $S_{elem}$ has the following properties:
\begin{enumerate}
\item From lemma \ref{l:1.4} we can deduce that the surface $S_{elem}$ is a tame translation surface with space of ends homeomorphic to $X$ and no planar ends. 

\item The families of marks $h_{j}\check{M}^{-j}$ and $M^{-j}$, for each $j\in\{1,...,|H|\}$ have (still) not been glued to other marks.
\item By construction, there is \emph{only one} end $[U^{elem}_{n}]_{n\in\mathbb{N}}$ of $S_{elem}$ having the following property: for every $n$ there exists $U_n^{elem}$ such that $U^{elem}_{n}\cap (h_{j}\check{M}^{-j} \cup M^{-j})$ are infinitely many marks. This end is defined by taking complements of large balls centered at the origin in the Euclidean plane $\mathbb{E}$ defining the decorated Loch Ness Monster $S(\tilde{\gamma})$. We will call this end the \emph{distinguished end} of the elementary piece $S_{elem}$. 
\item On the other hand, every other end $[U_{n}]_{n\in\mathbb{N}}\in Ends(S_{elem})\setminus\{[U^{elem}_{n}]_{n\in\mathbb{N}}\}$ satisfies that there exists an $n$ such that the intersection $U_{n}$ with the families of marks $\check{M}^{-j} \cup M^{-j}$ in $S(\tilde{\gamma})$ are empty, for each $j\in\{1,...,|H|\}$.
\end{enumerate}
\end{remark}

\begin{definicion}[\emph{Puzzle and assembled surface}]
\label{d:3.1}
Let $X$ be a closed subset of the Cantor set, $G$ a countable subgroup of $\rm GL_+(2,\mathbb{R})$ without contracting elements generated by $H$. For each $g\in G$ let $S_g:=g\cdot S_{elem}$ the affine copy of $S_{elem}$ obtained by postcomposing its translation atlas with the linear transformation defined by $g$. For each $j\in\{1,...,|H|\}$ we denote by $gh_{j}\check{M}^{-j}$ and $gM^{-j}$ the families of marks on $S_g$ given by the image of the families of marks $h_{j}\check{M}^{-j}$ and $M^{-j}$ via the affine diffeomorphism $\overline{g}:S_{elem}\to S_g$ (see \S \ref{SS:TSVG}). We define the \emph{puzzle} associated to the triplet $(X,G,H)$ as the set of marked surfaces\footnote{That is, on $S_g$ we consider the families of marks $gh_{j}\check{M}^{-j}$ and $gM^{-j}$.}:
\begin{equation}\label{eq:19}
\mathfrak{P}(X,G,H):=\{S_{g}:g\in G \}.
\end{equation}
\end{definicion}
The \emph{assembled surface} associated to the puzzle $\mathfrak{P}(X,G,H)$ is defined as:
\begin{equation}\label{eq:20}
S_{\mathfrak{P}}:= \bigcup\limits_{g\in G}S_g \bigg/ \sim
\end{equation}
where $\sim$ is the equivalent relation given by gluing marks as follows. Given an edge $(g,gh_{j})$ of the Cayley graph $Cay(G,H)$,
we glue\footnote{Remark that by construction marks we glue are indeed parallel.} for each $i\in\mathbb{N}$, the mark 
$gh_{j}\check{m}_{i}^{-j}\in gh_{j}\check{M}^{-j}\subset S_{g}$ to the mark $gh_{j}m_{i}^{-j}\in gh_{j}M^{-j}\subset S_{gh_{j}}$.

\begin{remark}
The pieces $S_g$ of the puzzle inherit all affine invariant properties from $S_{elem}$, namely:
\begin{enumerate}
\item $S_g$ is a tame translation surface without planar ends whose space of ends  is homeomorphic to $X$.
\item For every $g\in G$ there is \emph{only one} end  $[U^g_n]$ of $S_g$ having the following property: for every $n$ there exists $U^g_n$ such that $U^{g}_{n}\cap (gh_{j}\check{M}^{-j} \cup gM^{-j})$ are infinitely many marks. We will call this end the \emph{distinguished end} of $S_g$. 
\item Every other end $[U_{n}]_{n\in\mathbb{N}}\in Ends(S_{g})\setminus\{[U^{g}_{n}]_{n\in\mathbb{N}}\}$ satisfies that there exists an $n$ such that the intersection $U_{n}\cap(gh_{j}\check{M}^{-j} \cup gM^{-j})=\emptyset$, for each $j\in\{1,...,|H|\}$. We will call these kind of ends \emph{common ends}. 
\end{enumerate}
\end{remark}
Property (3) above tells us that every common end $[U_n]_{n\in\mathbb{N}}\in Ends(S_g)$ has a representative $U_n$ that avoids the families of marks used to assemble the surface $S_{\mathfrak{P}}$. Therefore we can \emph{embedd} $U_n$ into $S_{\mathfrak{P}}$ (using the identity) and induce an embedding for common ends:
\begin{equation}\label{eq:21}
\begin{array}{ccccc}
 i_{g} & :  & Ends(S_{g})\setminus\{[U^{g}_{n}]_{n\in\mathbb{N}}\} & \hookrightarrow & Ends(S_{\mathfrak{P}}) \\
          &     &[U_{n}]_{n\in\mathbb{N}}                                                   &\to & [U_{n}]_{n\in\mathbb{N}}.
\end{array}
\end{equation}
Remark that the end $i_{g}([U_{n}]_{n\in\mathbb{N}})$ of $S_{\mathfrak{P}}$ is not planar. The following proposition describes the space of ends of the assembled surface \emph{as a set}.
\begin{teorema}\label{pro:3.1}
The assembled surface $S_{\mathfrak{P}}$ is a tame translation surface without planar ends and its Veech group is $G$. Moreover:
\begin{equation}\label{eq:22}
Ends(S_{\mathfrak{P}})=\{[\widetilde{U}_{n}]_{n\in\mathbb{N}}\}\cup\left(\bigsqcup\limits_{g\in G}i_{g}(Ends(S_{g})\setminus\{[U^{g}_{n}]_{n\in\mathbb{N}}\}) \right). 
\end{equation}
\end{teorema}

We insist that the preceding is an equality in the category of sets, not in the category of topological spaces. In particular, this result says that all distinguished ends in the puzzle $\mathfrak{P}(X,G,H)$ \emph{merge into a single end} $\{[\widetilde{U}_{n}]_{n\in\mathbb{N}}\}$ when constructing $S_{\mathfrak{P}}$. We will call this end the \emph{secret end} of $S_{\mathfrak{P}}$. 


\begin{proof}
\emph{We begin by showing that the assembled surface $S_{\mathfrak{P}}$ is tame}. This follows from the following two facts:
\begin{itemize}
\item \emph{The surface $S_{\mathfrak{P}}$ is a complete metric space\footnote{W.r.t. the distance induced by the natural flat metric on $S_{\mathfrak{P}}$}.} Indeed, let $(x_n)_{n\in \mathbb{N}}\subset S_{\mathfrak{P}}$ be a Cauchy sequence. Lemma \ref{l:3.1} implies that the cost (in distance inside $S_{\mathfrak{P}}$) to scape from a piece $S_g$ of the puzzle is at least $\frac{1}{\sqrt{2}}$. Hence, the sequence $(x_n)$ is eventually contained in the the closure (in $S_{\mathfrak{P}}$) of the 
open subset: 
\[
U(g):= S_{g}\setminus \bigcup_{j=1}^{|H|}(gh_{j}\check{M}^{-j}\cup gM^{-j}),
\]
and this closure is, by construction, a complete metric  subspace of $S_{\mathfrak{P}}$.
\item \emph{The set of singularities of $S_{\mathfrak{P}}$ is discrete in $S_{\mathfrak{P}}$}. This follows from lemma \ref{l:3.1} and the fact that the set of singularities at each piece $S_g$ of the puzzle is discrete.
\end{itemize}
We now proof that the Veech group of $S_{\mathfrak{P}}$ is $G$. For every $g,g'\in G$ there is a natural affine diffeomorphism $f_{gg'}:S_g\to S_{g'g}$ whose differential is precisely $g'$. These transformations send parallel marks to parallel marks, therefore one can glue all $f_{gg'}$'s together to induce an affine diffeomorphism in the quotient $F_{g'}:S_{\mathfrak{P}}\to S_{\mathfrak{P}}$ whose differential is precisely $g'$. Since $g'$ was arbitrary we have that $G<\Gamma(S_{\mathfrak{P}})$. 
When constructing the elementary piece $S_{elem}$ we added in purpose a decorated Loch Ness Monster so that $S_{elem}$ has only one $6\pi$ singularity $x(Id)$ and only three saddle connections $\gamma_1$, $\gamma_2$, $\gamma_3$ issuing from it. Moreover, the holonomy vectors of these saddle connections are $\{\pm e_1, \pm e_2\}$. This implies that every piece $S_g$ in the puzzle $\mathfrak{P}(X,G,H)$ has \emph{only one} singularity $x(g)$ of total angle $6\pi$ only three saddle connections $\gamma_1$, $\gamma_2$, $\gamma_3$ issuing from it. The holonomy vectors of these are $\{\pm g\cdot e_1, \pm g\cdot e_2\}$. On the other hand suppose that an affine diffeomorphism $f\in{\rm Aff}_+(S_{\mathfrak{P}})$ sends $x(Id)\in S_{elem}=S_{Id}$ to $x(g)$. Its derivative $Df$ must then send $\{\pm e_1, \pm e_2\}$ to $\{\pm g\cdot e_1, \pm g\cdot e_2\}$ and have positive determinant. The only possibility is $Df=g$, therefore $\Gamma(S)<G$.

We address now equation (\ref{eq:22}). Given that this is the most technical part of the proof we divide our approach in three steps:

\textbf{I.} For any two different elements $g\neq g'$ in $G$ we have that the embeddings defined in (\ref{eq:21}) have disjoint images:
\begin{equation}
	\label{e:disj:ends}
i_{g}(Ends(S_{g})\setminus\{[U^{g}_{n}]_{n\in\mathbb{N}}\})\bigcap i_{g^{'}}(Ends(S_{g^{'}})\setminus\{[U^{g^{'}}_{n}]_{n\in\mathbb{N}}\})=\emptyset.
\end{equation}
Therefore $\bigsqcup\limits_{g\in G}i_{g}(Ends(S_{g})\setminus\{[U^{g}_{n}]_{n\in\mathbb{N}}\})$ is a subset of $Ends(S_{\mathfrak{P}})$.  Indeed, consider two ends  $[W_{n}]_{n\in\mathbb{N}}\in i_{g}(Ends(S_{g})\setminus\{[U^{g}_{n}]_{n\in\mathbb{N}}\})$ and $[Z_{n}]_{n\in\mathbb{N}}\in  i_{g^{'}}(Ends(S_{g^{'}})\setminus\{[U^{g^{'}}_{n}]_{n\in\mathbb{N}}\})$. Without loss of generality we can suppose that  $[W_{n}]_{n\in\mathbb{N}}$ and $[Z_{n}]_{n\in\mathbb{N}}$ are ends of the pieces $S_g$ and $S_{g^{'}}$, respectively. 
Given that $[W_{n}]_{n\in\mathbb{N}}$ is different from the distinguished $[U_n^g]_{n\in\mathbb{N}}$, there exists a representative $W_N$ which does not intersect any of the buffer surfaces in the decorated Loch Ness Monster forming $S_g$. Since path from $S_g$ to $S_{g'}$ has to go through one of these buffer surfaces, there exists a representative $Z_M$ such that $W_N\cap Z_M=\emptyset$. Therefore the ends  $[W_{n}]_{n\in\mathbb{N}}$ and $[Z_{n}]_{n\in\mathbb{N}}$ are disjoint in $Ends(S_{\mathfrak{P}})$.

\textbf{II.} There exists an end $[\widetilde{U}_{n}]_{n\in\mathbb{N}}$, that we will call \emph{the secret end}, in the complement of $\bigsqcup\limits_{g\in G}i_{g}(Ends(S_{g})\setminus\{[U^{g}_{n}]_{n\in\mathbb{N}}\})$  in $Ends(S_{\mathfrak{P}})$ . We construct this end inductively in what follows. First we choose an enumeration $G:=\{g_{1},...,g_{|G|}\}$ and $H:=\{h_{1},...,h_{|H|}\}$ for elements in $G$ and $H$ respectively. Since surfaces are $\sigma$-compact spaces, for every $g\in G$ there exist an exhaustion of $S_g=\bigcup_{n\in\mathbb{N}}gK_{n}$ by compact sets \emph{whose complements define the ends space of the surface}. More precisely, we can write 
\begin{equation}\label{eq:23}
S_g\setminus gK_{n} := gU^{n}_{1}\sqcup ...\sqcup gU^{n}_{k(n)} \sqcup ...\sqcup gU^{n}_{i_n},
\end{equation}
where each $gU^{n}_{k(n)}$ with $k(n)\in \{1,...,i_n\}$ is a connected component whose closure in $S_g$ is noncompact, but has compact boundary, and for every $k(n+1)\in\{1,...,i_{n+1}\}$ there exist $k(n)\in\{1,...,i_n\}$ such that $gU_{k(n)}^n \supset gU^{n+1}_{k(n+1)}$. In other words, the space $Ends(S_g)$ are all nested sequences $(gU^{n}_{k(n)})_{n\in \mathbb{N}}$. Without loss of generality, we can assume that  $[gU^n_1]_{n\in\mathbb{N}}$ defines the distinguished end $[U^{g}_{n}]_{n\in\mathbb{N}}$ of $S_g$ for all $g\in G$.
 Now consider the $g_{1}\in G$. By taking $g_1K_1$ big enough we have the following decomposition into connected components:
\[
S_{\mathfrak{P}}\setminus g_1K_1=\widetilde{U_1}\sqcup g_1U_2^1\sqcup\ldots\sqcup g_1U_{i_1}^1
\]
where  $g_1U_j^1\subset S_{g_1}$ and  $g_1U_j^1\cap U_n^{g_1}=\emptyset$ for all $j=2,\ldots,i_1$ and $n$ big enough. By definition, the connected component 
 $\widetilde{U}_{1}\subset S_{\mathfrak{P}}$ contains  $U^{g_{1}}_{1}=g_1U_1^1$, has compact boundary and is not planar for $U^{g_{1}}_{1}\subset \widetilde{U}_{1}$ has infinite genus. To define $\widetilde{U}_n$ for $n>1$ consider the "first" $n$ elements $g_{1},...,g_{n}$ in $G$. By taking $g_jK_n$ big enough we have the following decomposition in connected components:
  \[
      S_{\mathfrak{P}}\setminus \bigcup_{k=1}^{n}g_{k}K_{n}=\widetilde{U}_{n}\sqcup \left( \bigsqcup_{k=1}^{n}\left(g_{k}U^{n}_{2}\sqcup...\sqcup g_kU^{n}_{k(n)}\sqcup...\sqcup g_{k}U^{n}_{i_n}\right)\right).
      \] 
where $g_kU_{k(n)}^n\subset S_{g_k}$ and $g_kU_{k(n)}^n\cap U_m^{g_k}=\emptyset$ 
for all $k=1,\ldots,n$, $k(n)=2,\ldots,i_n$ and 
$m$ big enough. By definition, the connected component $\widetilde{U}_n$ contains $\bigcup_{k=1}^{n}U^{g_{k}}_{n}$, has compact boundary and is not planar. Moreover $\widetilde{U}_{n-1}\supset\widetilde{U}_{n}$ and, since $\{gK_n\}_{n\in\mathbb{N}}$ is an exhaustion of $S_g$, for every compact subset $K$ in $S_{\mathfrak{P}}$, there exists $N$ such that $\widetilde{U}_N\cap K=\emptyset$. In other words, $[\widetilde{U}_n]_{n\in\mathbb{N}}$ is an element of $Ends(S_{\mathfrak{P}})$ and is not planar. In Figure (\ref{Figure13}) we depict this secret end when $G$ is an infinite cyclic group. 

Now let $[V_n]_{n\in\mathbb{N}}\in\bigsqcup_{g\in G}i_{g}(Ends(S_{g})\setminus\{[U^{g}_{n}]_{n\in\mathbb{N}}\})$. Without loss of generality, we can suppose that $[V_n]_{n\in\mathbb{N}}$ is an element in $Ends(S_{g_k})\setminus[U^{g_k}_n]_{n\in\mathbb{N}}$ for some $g_k\in G$.  Given that $[V_n]_{n\in\mathbb{N}}$ is not a distinguished end, there exist $m,l\in\mathbb{N}$ such that $V_m\subset S_{g_k}\setminus g_kK_l$ and $V_m$ is disjoint from all buffer surfaces in the decorated Loch Ness Monster in $S_{g_k}$, hence $V_m\cap S_{g'}=\emptyset$ for every $g'\neq g$. This implies that $[\widetilde{U}_n]_{n\in\mathbb{N}}$ cannot belong to $\bigsqcup_{g\in G}i_{g}(Ends(S_{g})\setminus\{[U^{g}_{n}]_{n\in\mathbb{N}}\})$. 

\textbf{III.}The secret end $[\widetilde{U}_{n}]_{n\in\mathbb{N}}$ and $\bigsqcup\limits_{g\in G}i_{g}(Ends(S_{g})\setminus\{[U^{g}_{n}]_{n\in\mathbb{N}}\})$ \emph{is all there is} in $Ends(S_{\mathfrak{P}})$. Consider $[W_{n}]_{n\in\mathbb{N}}$ an end of $S_{\mathfrak{P}}$. For every $n\in \mathbb{N}$ there exist $l(n)\in \mathbb{N}$ such that
\[
W_{l(n)}\subset S_{\mathfrak{P}}\setminus \bigcup_{k=1}^{n}g_{k}K_{n}=\widetilde{U}_{n}\sqcup\left(\bigsqcup_{k=1}^{n} \left(g_{k}U_{2}^{n}\sqcup...\sqcup g_{k}U^{n}_{i_n}\right)\right). 
\]
There are two cases to consider. First suppose that there exist $N\in\mathbb{N}$, such that $W_{l(N)}\subset g_{k}U^{N}_{j}$ for some $(k,j)\in\{1,...,n\}\times \{1,...,i_n\}$. 
In this situation we have $[W_{n}]_{n\in\mathbb{N}}\in i_{g}(Ends(S_{g})\setminus\{[U^{g}_{n}]_{n\in\mathbb{N}}\})$ for some $g\in G$. 
Suppose now that for every $n\in\mathbb{N}$, there exists $l(n)$ such that $W_{l(n)}\subset \widetilde{U}_{n}$. If we fix $n$, there exist $k(n)$ such that $\widetilde{U}_{k(n)}\cap \partial W_{n}=\emptyset$. Hence, the open subset $\widetilde{U}_{k(n)}\subset S_{\mathfrak{P}}$ is contained in a connected component of $S_{\mathfrak{P}}\setminus \partial W_{n}$. Given our assumption, there exist $l(k(n))\in \mathbb{N}$, such that $W_{l(k(n))}\subset\widetilde{U}_{k(n)}$. Now, since $\widetilde{U}_{k(n)}$ is connected we have that either $W_{n}\subset W_{l(k(n))}$ or $W_{l(k(n))}\subset W_{n}$ implies that the connected component of $S_{\mathfrak{P}}\setminus \partial W_{n}$ containing $\widetilde{U}_{k(n)}$ is precisely $W_n$. We conclude then that for every $n$ there exists $k(n)$ such that $\widetilde{U}_{k(n)}\subset W_{n}$ and hence the ends $[\widetilde{U}_n]_{n\in\mathbb{N}}$ and $[W_n]_{n\in\mathbb{N}}$ are the same.

To finish the proof remark that, by construction, all ends in (\ref{eq:22})
have infinite genus.
\end{proof}

\begin{figure}[h!]
  \centering
  \includegraphics[scale=0.35]{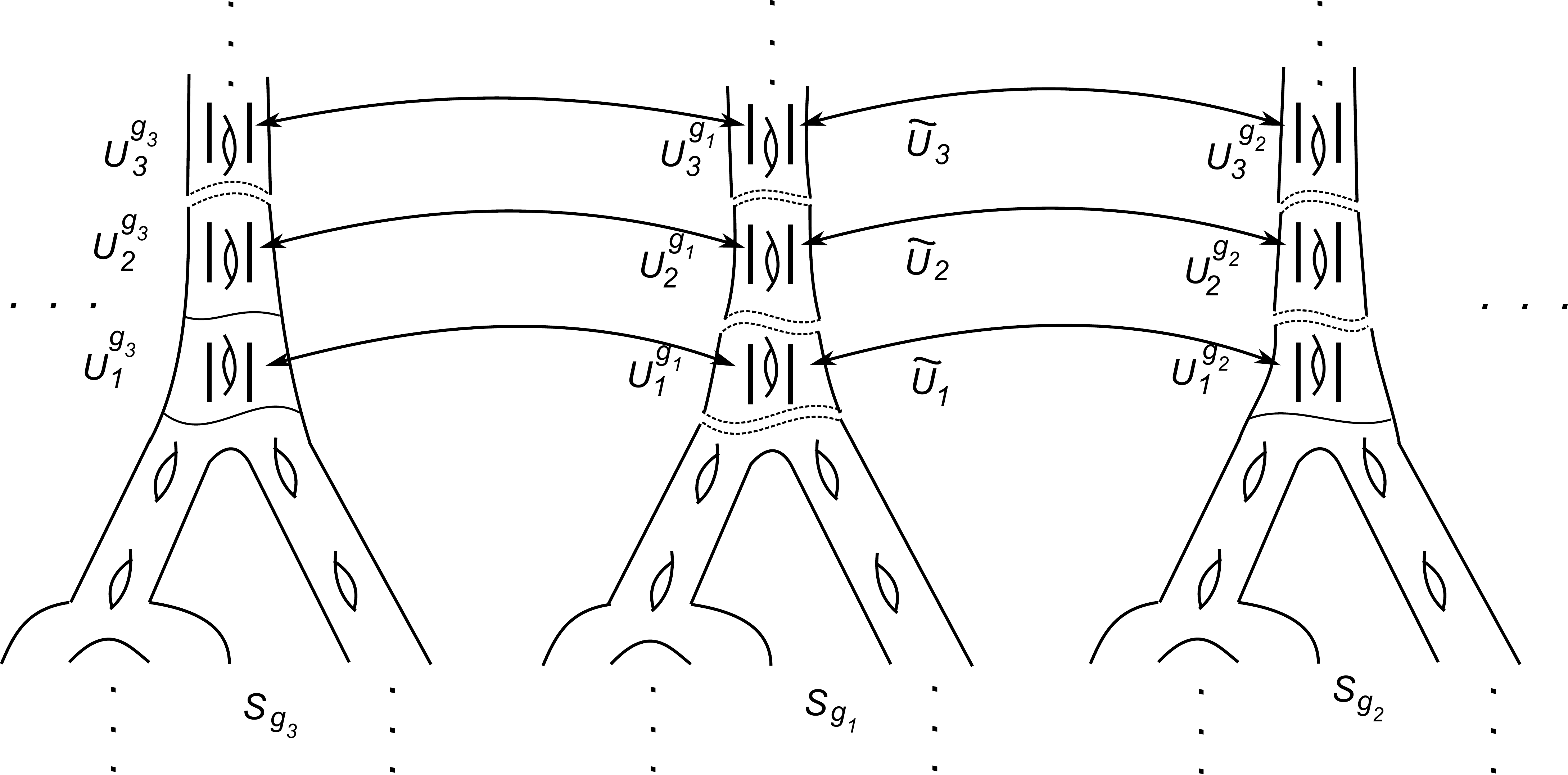}
  \caption{\emph{Secret end of $S_{\mathfrak{P}}$ when $G=\mathbb{Z}$.}}
  \label{Figure13}
\end{figure}



\subsection{Proof of the Theorem \ref{T:AF}.} 
Let $2^\omega$ denote the Cantor set and consider the puzzle $\mathfrak{P}(2^{\omega},G,H)$. Following theorem \ref{pro:3.1}, it is sufficient to prove that $Ends(S_\mathfrak{P})$ has no isolated points. Let then $W^*$ be an open neighbourhood of $[U_n]_{n\in\mathbb{N}}\in Ends(S_\mathfrak{P})$. From equation (\ref{eq:22}) we have two cases to consider. First, suppose that there exist $g\in G$ such that $[U_{n}]_{n\in\mathbb{N}}\in i_{g}(Ends(S_{g})\setminus\{[U^{g}_{n}]_{n\in\mathbb{N}}\})$. In this case there exist an open subset $V$ of $S_\mathfrak{P}$ with compact boundary such that $U_l\subset V\subset W\cap S_g$. From remark \ref{rem:3.1}, we know that $Ends(S_g)$ is homeomorphic to $2^\omega$. Hence $V^*\setminus [U_n]_{n\in\mathbb{N}}\subset W^*\setminus [U_n]_{n\in\mathbb{N}}$ is not empty. Now suppose that $[U_{n}]_{n\in\mathbb{N}}$ is equivalent to the secret end $[\widetilde{U}_{n}]_{n\in\mathbb{N}}$. We know that there exists $U_l\subset W$ and, by construction, $U_k\subset U_l$ such that $U_k^*\subset W^*$ contains $Ends(S_g)\setminus [U_n^g]_{n\in\mathbb{N}}$, for some $g\in G$. Since the latter is homeomorphic to $2^\omega$, we have that $W^*\setminus[U_n]_{n\in\mathbb{N}}$ is not empty.
\rightline{$\Box$}

\subsection{Proof of the Theorem \ref{T:AN}.}  Consider the ordinal number $\omega^k+1$, for a fixed $k\in\mathbb{N}$. Following theorem \ref{pro:3.1} it is sufficient to prove that $Ends(S_\mathfrak{P})$ is homeomorphic to $\omega^k+1$. Roughly speaking the idea of the proof is the following: first we choose properly the path $\tilde{\gamma}$ in
construction \ref{cons:3.3}, so that the k-th iterate of the Cantor-Bendixon derivative on $Ends(S_g)$ is precisely the distinguished end $[U_n^g]$ for all $g\in G$. Using this and some properties of the secret end for this particular case, we will define a countable topological space $Y_k$. Finally, we will prove that $Ends(S_\mathfrak{P})$ is homeomorphic to $Y_k$ and that the characteristic system of $Y_k$ is precisely $(k,1)$.

Fix a topological embedding $\omega^k+1\hookrightarrow Ends(T2^\omega)$, let 
$T_{\omega^k+1}$ be the graph with ends space homeomorphic to ordinal number $\omega^k+1$ given by lemma \ref{l:1.1} and  $\mathfrak{T}_{\omega^k+1}$ the countable family of paths given by lemma \ref{l:1.2}. Given that $\omega^k+1$ is countable, the sets $Ends(T_{\omega^k+1})$ and $\mathfrak{T}_{\omega^k+1}$ are in bijection (see corollary \ref{c:1.3}). Let then $\tilde{\gamma}\in\mathfrak{T}_{\omega^k+1}$ be the infinite path corresponding to the only point left in $Ends(T_{\omega^k+1})$ after the k-th iteration of the Cantor-Bendixon derivative.  We perform then construction \ref{cons:3.3} of the elementary piece of the Puzzle $\mathfrak{P}(\omega^k+1,G,H)$ choosing the infinite path $\tilde{\gamma}\in \mathfrak{T}_{\omega^k+1}$, associated to the decorated Loch Ness Monster, as above. 
With this choice we assure that the k-th iteration of the Cantor-Bendixon derivative on the space of ends of every piece $S_g$ of the puzzle $\mathfrak{P}(\omega^k+1,G,H)$  is precisely the distinguished end $[U_n^g]_{n\in\mathbb{N}}$. 
Let $U\subset S_\mathfrak{P}$ be a connected open subset with compact boundary defining an open neighborhood $U^*$ of the secret end  $[\widetilde{U}_{n}]_{n\in\mathbb{N}}$ in $Ends(S_\mathfrak{P})$. Then there  exist a finite subset $G(U^{\ast})\subset G$ such that:
\begin{enumerate}
\item For every $g\in G(U^{\ast})$, $\partial U\cap S_{g}\neq \emptyset$ and $i_{g}(Ends(S_{g})\setminus\{[U^{g}_{n}]_{n\in\mathbb{N}}\})$ is not properly contained in $U^*$. 
\item  For every element $g\in G\setminus G(U^{\ast})$, $\partial U \cap S_{g}=\emptyset$ and  $i_{g}(Ends(S_{g})\setminus\{[U^{g}_{n}]_{n\in\mathbb{N}}\})\subset U^{\ast}$.
\end{enumerate}
To illustrate this properties consider the non-trivial example $G=\mathbb{Z}$ and $Ends(S_g)$ homeomorphic to $\omega+1$ depicted in Figure \ref{Fig:SpaceX}.
\begin{figure}[h!]
  \centering
  \includegraphics[scale=0.55]{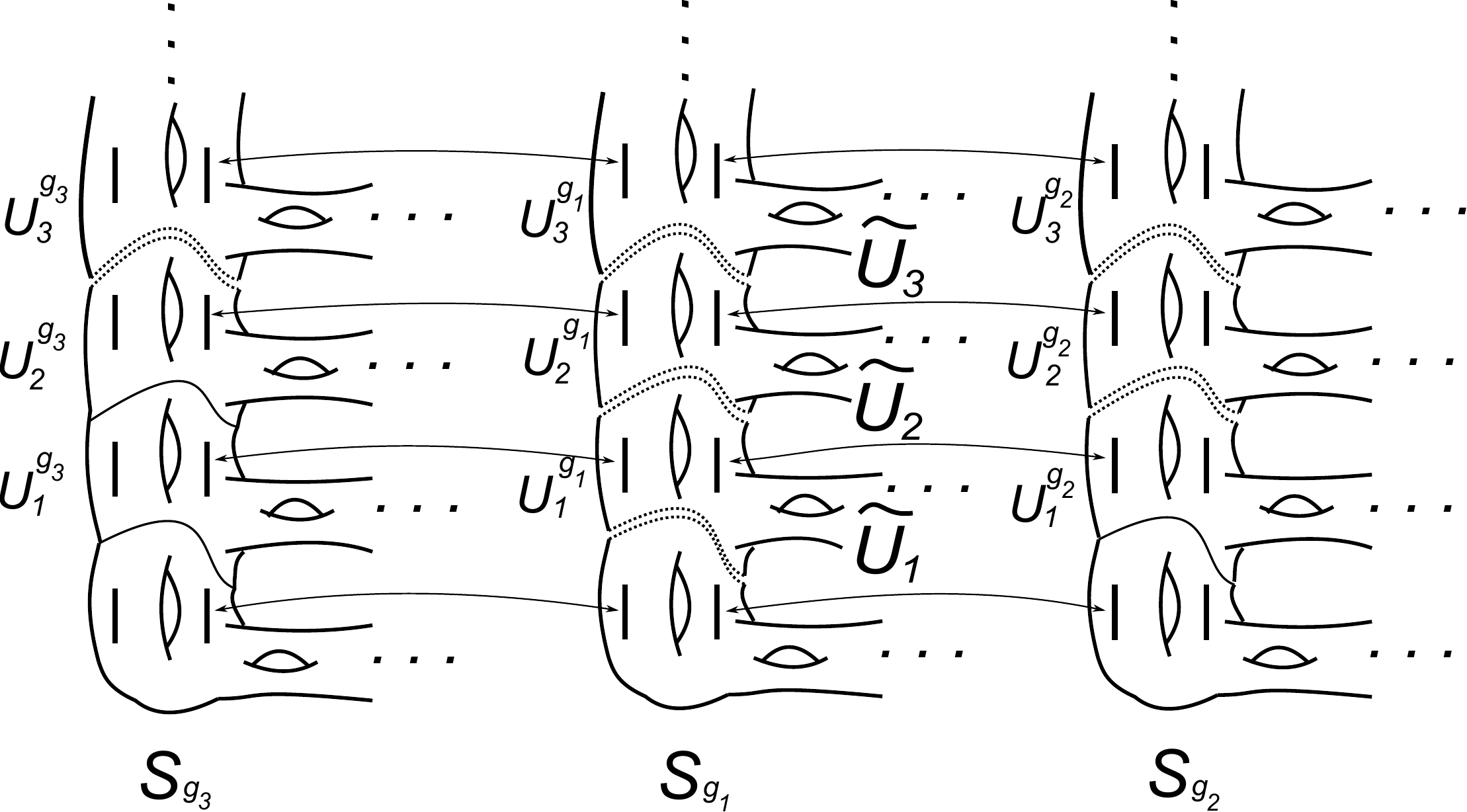}
  \caption{The surface $S_{\mathfrak{P}}$ when $G=\mathbb{Z}$ and $Ends(S_g)$ is homeomorphic to $\omega+1$.}
  \label{Fig:SpaceX}
\end{figure}

For every $g\in G$ let $\omega^{k}_{g}$ be a copy of the ordinal number $\omega^{k}$ and define
\begin{equation}\label{eq:27}
Y_k:=\{y\}\cup\left(\bigsqcup_{g\in G}\omega^{k}_{g}\right),
\end{equation}
where $y$ is just an abstract point. We endow $Y_k$ with a topology as follows. Let $U^*$ be an open neighbourhood of the secret end  $[\widetilde{U}_n]_{n\in\mathbb{N}}$  of $S_{\mathfrak{P}}$ and $G(U^{\ast})\subset G$ a finite subset with $m(U^*)$ elements defined as above. For every   $\{\gamma_{g_{j_1}},\dots,\gamma_{g_{j_{m(U^{\ast})}}}\}\in \prod_{g\in G(U^*)} \omega^{k}_g$, we define :
\begin{equation}
	\label{Eq:BasicOpen}
W(G(U^{\ast}),\{\gamma_{g_{j_1}},\dots,\gamma_{g_{j_{m(U^{\ast})}}}\}):=\{y\}\cup\left(\bigsqcup_{n=1}^{m(U*)}\{\beta\in\omega_{g_{j_n}}^{k}:\beta\succ\gamma_{g_{j_n}}\}\right)\cup\left(\bigsqcup_{g\in G\setminus G(U^{\ast})}\omega^{k}_{g}\right)\subset Y_k.
\end{equation}

Then $\mathcal{B}:=\{W(G(U^{\ast}),\{\gamma_{g_{j_1}},\dots,\gamma_{g_{j_{m(U^{\ast})}}}\})\}\cup\{W: W$ is an open subset of $\omega^{k}_{g}$ for any $g\in G\}$ is the basis for the topology of $Y_k$. Remark that $Y_k$ is a countable Hausdorff space with respect to this topology.

We now prove that $Y_k$ is homeomorphic to $Ends(S_\mathfrak{P})$. By the way we chose the infinite path $\tilde{\gamma}\in \mathfrak{T}_{\omega^k+1}$, for every $g\in G$ there exists a homeomorphism $f_g:i_g(Ends(S_g)\setminus\{[U^g_n]_{n\in\mathbb{N}}\})\to \omega_g^k$. Define $F:Ends(S_{\mathfrak{P}})\to Y_k$  as
\[
[V_{n}]_{n\in\mathbb{N}}  \to  \left\{
                                             \begin{array}{cl}
                                               f_{g}([V_{n}]_{n\in\mathbb{N}}) & \hbox{if $[V_{n}]_{n\in\mathbb{N}}\in i_g(Ends(S_g)\setminus\{[U^g_{n}]_{n\in\mathbb{N}}\})$ for some $g\in G$,} \\
                                               y & \hbox{if $[V_{n}]_{n\in\mathbb{N}}$ is the secret end $[\widetilde{U}_{n}]_{n\in\mathbb{N}}$.}
                                             \end{array}
                                           \right.
\]
To prove that $F$ is a homeomorphism it is sufficient to prove that it is continuous at the secret end $[\widetilde{U}_{n}]_{n\in\mathbb{N}}$, for every continuous map from a compact and Hausdorff space into a Hausdorff space is a closed map (see \cite[p. 226]{Dugu}). Consider an open neighbourhood $W(G(U^{\ast}),\{\gamma_{g_{j_1}},\dots,\gamma_{g_{j_{m(U^{\ast})}}}\})$ of $y$ as in (\ref{Eq:BasicOpen}). Since $G(U^*)$ has $m(U^*)<\infty$ elements, there exists a compact set $K$ in $S_\mathfrak{P}$ such that its complement $U_K$ satisfies:
\begin{enumerate}
\item $U_K^*\cap i_{g_{j_n}}(Ends(S_{g_{j_n}})\setminus \{[U_n^{g_{j_n}}]_{n\in\mathbb{N}}\})\subset F^{-1}(\{\beta\in\omega_{g_{j_n}}^{k}:\beta\succ\gamma_{g_{j_n}}\})$ for all $n=1,\ldots,m(U^*)$ 
\item $[\widetilde{U}_n]_{n\in\mathbb{N}}\in U_K^*$.
\end{enumerate}
Therefore $F(U_K^*)\subset W(G(U^{\ast}),\{\gamma_{g_{j_1}},\dots,\gamma_{g_{j_{m(U^{\ast})}}}\})$ as desired and hence $F$ is continuous. 

We claim that the characteristic system of $Y_k$ is $(k,1)$. To see this first remark that for each $g\in G$, the copy $\omega_g^k$ of $\omega^k$ figuring in the right-hand side of (\ref{eq:27})  is (topologically) embedded in $Y_k$. On the other hand, given that $\omega^k$ is not a limit ordinal, $y$ is a limit  point of the subset $\sqcup_{g\in G}\omega_g^k$. These two facts combined imply that
the set of accumulation points of $Y_k$ is precisely $Y_{k-1}$ for every $k\geq 2$ and just the singleton $\{y\}$ when $k=1$. Hence the $(k-1)$-th iteration of the Cantor-Bendixon derivative on $Y_k$ yields $Y_1$, which has characteristic system $(1,1)$. This implies that the characteristic system of $Y_k$ is precisely $(k,1)$ and the proof is complete.

\rightline{$\Box$}

\subsection{Proof of Theorem \ref{T:NN}.} 
Following theorem \ref{pro:3.1}, it is sufficient to prove that $Ends(S_\mathfrak{P})$ is homeomorphic to $B\sqcup U$, where by hypothesis $B$ is homeomorphic to the Cantor set and $U$ is a countable discrete set of points such that $\partial U = u\in B$.  As in the proof of theorem \ref{T:AN}, we deal first with $\mathfrak{P}(B\sqcup U,G,H)$'s elementary piece. From corollary \ref{c:1.4} we know that there exists a path $\tilde{\gamma}\in \mathfrak{T}_{B\sqcup U}$ 
that defines the end of $T_{B\sqcup U}$ 
corresponding to $u$, the boundary of $U$ in $B\sqcup U$. We perform construction \ref{cons:3.3} of the elementary piece taking the decorated Loch Ness Monster as $S(\tilde{\gamma})$, where $\tilde{\gamma}$ is chosen just as mentioned before. This way, for every $g\in G$ we have that $Ends(S_g)=B_g\sqcup U_g$, with $B_g$ homeomorphic to the Cantor set, $U_g$ discrete, countable and $\partial U_g$ is equal to the distinguished end $[U_n^g]_{n\in\mathbb{N}}$. On the other hand, given that for every $g\in G$ the map $i_g:Ends(S_g)\setminus\{[U_n^g]_{n\in\mathbb{N}}\}\to Ends(S_\mathfrak{P})$ defined in (\ref{eq:21}) is an embedding, $Ends(S_\mathfrak{P})$ is uncountable and:
\begin{equation}
	\label{equationRef2}
i_{g}(Ends(S_{g})\setminus\{[U^{g}_{n}]_{n\in\mathbb{N}}\}) =i_{g} (B_{g} \setminus \{ [U^{g}_{n}]_{n\in\mathbb{N}} \}) \sqcup i_{g}(U_{g}).
\end{equation}
By the Cantor-Bendixon theorem \ref{t:1.2}, we can write $Ends(S_{\mathfrak{P}})=B_\mathfrak{P}\sqcup U_\mathfrak{P}$, where $B_\mathfrak{P}$ is homeomorphic to the Cantor set and $U_\mathfrak{P}$ is discrete and countable. From theorem \ref{t:1.3} it is sufficient to show that $\partial U_\mathfrak{P}$ is just a point to finish the proof. We achieve this in what follows. 

Recall that :



\begin{equation}
	\label{eq:descripEnds}
Ends(S_{\mathfrak{P}})=B_{\mathfrak{P}}\sqcup U_{\mathfrak{P}}=\{[\widetilde{U}_{n}]_{n\in\mathbb{N}}\}\cup\left(\bigsqcup_{g\in G}i_{g}(Ends(S_{g})\setminus\{[U^{g}_{n}]_{n\in\mathbb{N}}\}) \right). 	
\end{equation}
We claim that $\partial U_{\mathfrak{P}}=\{[\widetilde{U}_{n}]_{n\in\mathbb{N}}\}
$. First remark that, by the choices we made, the secret end of $S_\mathfrak{P}$ is not an isolated point of $Ends(S_\mathfrak{P})$, hence $[\widetilde{U}_{n}]_{n\in\mathbb{N}}\in B_{\mathfrak{P}}$. From (\ref{equationRef2}) and (\ref{eq:descripEnds}) we obtain:
\begin{equation}
	\label{eq:37}
B_{\mathfrak{P}}=\{[\widetilde{U}_{n}]_{n\in\mathbb{N}}\}\sqcup \left(\bigsqcup\limits_{g\in G}i_{g}(B_{g}\setminus\{[U^{g}_{n}]_{n\in\mathbb{N}}\}) \right) \text{ and } U_{\mathfrak{P}}=\bigsqcup\limits_{g\in G} i_{g}(U_{g}).
\end{equation}
Every neighbourhood of the secret end $[\widetilde{U}_{n}]_{n\in\mathbb{N}}$ intersects $i_g(U_g)$ for infinitely many $g\in G$, therefore $[\widetilde{U}_{n}]_{n\in\mathbb{N}}\in\partial U_\mathfrak{P}$. Now let's prove by contradiction that if $[V_n]_{n\in\mathbb{N}}\in\partial U_\mathfrak{P}$, then $[V_n]_{n\in\mathbb{N}}=[\widetilde{U}_{n}]_{n\in\mathbb{N}}$.  
By (\ref{eq:37}), if  $[V_n]_{n\in\mathbb{N}}\neq [\widetilde{U}_{n}]_{n\in\mathbb{N}}$ then $[V_n]_{n\in\mathbb{N}}\in i_{g}(B_{g}\setminus\{[U^{g}_{n}]_{n\in\mathbb{N}}\})$ for some $g\in G$. Without loss of generality we can suppose that $[V_n]_{n\in\mathbb{N}}$ is actually an end in $B_g\setminus\{[U_n^g]_{n\in\mathbb{N}}\}\subset Ends(S_g)$. Therefore there exist $l,n\in\mathbb{N}$ such that $V_l\cap U_n^g=\emptyset$. But then 
the open set of $S_\mathfrak{P}$ given by $V_l^*$ is an open neighbourhood of $[V_n]_{n\in\mathbb{N}}$ which lies in the complement of $U_\mathfrak{P}=\bigsqcup\limits_{g\in G} i_{g}(U_{g})$, which contradicts $[V_{n}]_{n\in\mathbb{N}}\in\partial U_\mathfrak{P}$.

\rightline{$\Box$}

\end{document}